\documentclass[10pt]{amsart}

\usepackage{amsmath,amsfonts,amssymb,amsxtra,setspace,xspace,graphicx,lmodern,fourier,mathrsfs}
\usepackage[colorlinks=true]{hyperref}
\hypersetup{urlcolor=blue, citecolor=red, linkcolor=blue}
\usepackage[usenames,dvipsnames]{color}

\newtheorem{theorem}{Theorem}
\newtheorem{lemma}[theorem]{Lemma}

\newtheorem{corollary}[theorem]{Corollary}
\newtheorem{proposition}[theorem]{Proposition}

\newcommand{\haref}[2]{\href{https://hal.archives-ouvertes.fr/hal-#1}{hal-#1} \& \href{http://arxiv.org/abs/#2}{arXiv:#2}}

\newcommand{\be}[1]{\begin{equation}\label{#1}}
\newcommand{\ee}{\end{equation}}
\renewcommand{\(}{\left(}
\renewcommand{\)}{\right)}
\newcommand{\R}{{\mathbb R}}

\renewcommand{\P}{\mathsf P}

\newcommand{\sphere}{{\mathbb S^{d-1}}}

\newcommand{\ird}[1]{\int_{\R^d}{#1}\,dx}

\newcommand{\nrm}[2]{\left\|#1\right\|_{#2}}
\newcommand{\D}{\mathsf D_\alpha\kern0.5pt}
\newcommand{\DD}[1]{\D #1}
\newcommand{\Dstar}{\mathsf D_\alpha^*\kern1pt}

\newcommand{\iBR}[1]{\int_{B_R}{#1}\,dx}
\newcommand{\idBR}[1]{\int_{\partial B_R}{#1}\,d\sigma}
\newcommand{\iwR}[1]{\int_{B_R}{#1}\,d\mu_n}
\newcommand{\iwrd}[1]{\int_{\R^d}{#1}\,d\mu_n}

\newcommand{\isph}[1]{\int_{\sphere}#1\,d\omega}
\newcommand{\scal}[2]{\langle#1,#2\rangle}
\newcommand{\Scal}[2]{\langle\!\langle#1,#2\rangle\!\rangle}

\newcommand{\C}{{\mathsf C}}
\newcommand{\Balpha}{\mathcal B_\alpha}

\usepackage{color}

\title[Nonlinear flows, boundary terms \& linearization]{Interpolation inequalities, nonlinear flows, boundary terms, optimality and linearization}

\author[J.~Dolbeault]{Jean Dolbeault}
\address{\hspace*{-12pt}J.~Dolbeault: CEREMADE (CNRS UMR n$^\circ$ 7534), PSL research university, Universit\'e Paris-Dauphine, Place de Lattre de Tassigny, 75775 Paris 16, France}
\email{dolbeaul@ceremade.dauphine.fr}

\author[M.J.~Esteban]{Maria J.~Esteban}
\address{\hspace*{-12pt}M.J.~Esteban: CEREMADE (CNRS UMR n$^\circ$ 7534), PSL research university, Universit\'e Paris-Dauphine, Place de Lattre de Tassigny, 75775 Paris 16, France}
\email{esteban@ceremade.dauphine.fr}

\author[M.~Loss]{Michael Loss}
\address{\hspace*{-12pt}M.~Loss: School of Mathematics, Georgia Tech, 686 Cherry St., Atlanta, GA, 30332-0160, USA}
\email{loss@math.gatech.edu}

\begin{document}
\begin{abstract} This paper is devoted to the computation of the asymptotic boundary terms in entropy methods applied to a fast diffusion equation with weights associated with Caffarelli-Kohn-Nirenberg interpolation inequalities. So far, only elliptic equations have been considered and our goal is to justify, at least partially, an extension of the \emph{carr\'e du champ} / Bakry-Emery / R\'enyi entropy methods to parabolic equations. This makes sense because evolution equations are at the core of the heuristics of the method even when only elliptic equations are considered, but this also raises difficult questions on the regularity and on the growth of the solutions in presence of weights.

We also investigate the relations between the optimal constant in the entropy -- entropy production inequality, the optimal constant in the information -- information production inequality, the asymptotic growth rate of generalized R\'enyi entropy powers under the action of the evolution equation and the optimal range of parameters for symmetry breaking issues in Caffarelli-Kohn-Nirenberg inequalities, under the assumption that the weights do not introduce singular boundary terms at $x=0$. These considerations are new even in the case without weights. For instance, we establish the equivalence of carr\'e du champ and R\'enyi entropy methods and explain why entropy methods produce optimal constants in entropy -- entropy production and Gagliardo-Nirenberg inequalities in absence of weights, or optimal symmetry ranges when weights are present.\end{abstract}


\maketitle
\thispagestyle{empty}
\vspace*{-0.75cm}

\section{Introduction}\label{Sec:Intro}

In this paper we consider the Gagliardo-Nirenberg inequality
\be{GN}
\nrm{\nabla w}2^\theta\,\nrm w{q+1}^{1-\theta}\ge\mathsf C_{\mathrm{GN}}\,\nrm w{2q}\quad\forall\,w\in C_0^\infty(\R^d)
\ee
in relation with the nonlinear diffusion equation in $\R^d$, $d\ge1$,
\be{poro}
\frac{\partial v}{\partial t}=\Delta v^m\,,\quad(t,x)\in\R^+\times\R^d\,,
\ee
in the fast diffusion regime $m\in[1-1/d,1)$.

We also consider more general interpolation inequalities with weights. With the norm defined by
\[
\nrm w{q,\gamma}:=\(\ird{|w|^q\,|x|^{-\gamma}}\)^{1/q}\,,
\]
which extends the case without weight $\nrm wq=\nrm w{q,0}$, let us consider the family of \emph{Caffarelli-Kohn-Nirenberg interpolation inequalities} given by
\be{CKN}
\nrm w{2p,\gamma}\le\C_{\beta,\gamma,p}\,\nrm{\nabla w}{2,\beta}^\vartheta\,\nrm w{p+1,\gamma}^{1-\vartheta}
\ee
in a suitable functional space $\mathrm H^p_{\beta,\gamma}(\R^d)$ obtained by completion of smooth functions with compact support in $\R^d\setminus\{0\}$, w.r.t.~the norm given by $\|w\|^2:=(p_\star-p)\,\nrm w{p+1,\gamma}^2+\nrm{\nabla w}{2,\beta}^2$. Here $\C_{\beta,\gamma,p}$ denotes the optimal constant, the parameters $\beta$, $\gamma$ and~$p$ are subject to the restrictions
\be{parameters}
d\ge2\,,\quad\gamma-2<\beta<\frac{d-2}d\,\gamma\,,\quad\gamma\in(-\infty,d)\,,\quad p\in\(1,p_\star\right]\quad\mbox{with}\,p_\star:=\frac{d-\gamma}{d-\beta-2}
\ee
and the exponent $\vartheta$ is determined by the scaling invariance, \emph{i.e.},
\[\label{theta}
\vartheta=\frac{(d-\gamma)\,(p-1)}{p\,\big(d+\beta+2-\,2\,\gamma-p\,(d-\beta-2)\big)}\,.
\]
These inequalities have been introduced, among other inequalities, by L.~Caffarelli, R.~Kohn and L.~Nirenberg in~\cite{Caffarelli-Kohn-Nirenberg-84}. The evolution equation associated with~\eqref{CKN} is the weighted nonlinear diffusion equation
\be{FD}
v_t=|x|^\gamma\,\nabla\cdot\(\,|x|^{-\beta}\,\nabla v^m\)\,,\quad(t,x)\in\R^+\times\R^d\,,
\ee
with exponent $m=\frac{p+1}{2\,p}\in[m_1,1)$ where
\[\label{m1}
m_1:=\frac{2\,d-2-\beta-\gamma}{2\,(d-\gamma)}\,.
\]
Details about the existence of solutions for the above evolution equation and their properties can be found in~\cite{1602}.

Our first goal is to give a proof of~\eqref{CKN} with an integral remainder term using~\eqref{FD} whenever the optimal function in~\eqref{CKN} is radially symmetric. This requires some parabolic estimates. As in the elliptic proof of~\eqref{CKN} given in~\cite{DEL2015} and~\cite{DELM}, the main difficulty arises from the justification of the integrations by parts. We also investigate why the method provides the optimal constant in~\eqref{GN} and the optimal range of symmetry in~\eqref{CKN}.

\subsection{The symmetry breaking issue}

Equality in~\eqref{CKN} is achieved by Aubin-Talenti type functions
\[
w_\star(x)=\(1+|x|^{2+\beta-\gamma}\)^{-1/(p-1)}\quad\forall\,x\in\R^d
\]
if we know that \emph{symmetry} holds, that is, if we know that the equality is achieved among radial functions. In this case it is not very difficult to check that $w_\star$ is the unique radial critical point, up to the transformations associated with the invariances of the equation. Of course, any element of the set of functions generated by the dilations and the multiplication by an arbitrary constant is also optimal if $w_\star$ is optimal. Conversely, there is \emph{symmetry breaking} if equality in~\eqref{CKN} is not achieved among radial functions.

Deciding whether symmetry or symmetry breaking holds is a central problem in physics, and it is also a difficult mathematical question. It is well known that symmetric energy functionals may have states of lowest energy that may or may not have these symmetries. In our example~\eqref{CKN} the weights are radial and the functional is invariant under rotation. In the language of physics, a \emph{broken} symmetry means that the symmetry group of the minimizer is strictly smaller than the symmetry group of the functional. For computing the optimal value of the functional it is of great advantage that an optimizer is symmetric. The optimal constant $\C_{\beta,\gamma,p}$ can then be explicitly computed in terms of the $\Gamma$ function. Otherwise, this is a difficult question which has only numerical solutions so far and involves a delicate energy minimization as shown in~\cite{MR3043639,MR3168260}. In other contexts the breaking of symmetry leads to various interesting phenomena and this is why it is important to decide what symmetry types, if any, an optimizer has. Our problem is a model case, in which homogeneity and scaling properties are essential to obtain a clear-cut answer to these symmetry issues.

To show that symmetry is broken in~\eqref{CKN}, one can minimize the associated functional in the \emph{class of symmetric functions} and then check whether the value of the functional can be lowered by perturbing the minimizer away from the symmetric situation. This is the standard method, and it has been used to establish that \emph{symmetry breaking} occurs in~\eqref{CKN}~if
\be{set-symm-breaking}
\gamma<0\quad\mbox{and}\quad\beta_{\rm FS}(\gamma)<\beta<\frac{d-2}d\,\gamma
\ee
where
\[
\beta_{\rm FS}(\gamma):=d-2-\sqrt{(\gamma-d)^2-4\,(d-1)}\,.
\]
In the critical case $p=p_\star$, the method was implemented by F.~Catrina and Z.-Q.~Wang in~\cite{Catrina-Wang-01}, and the sharp result was obtained by V.~Felli and M.~Schneider in~\cite{MR2017240}. The same condition was recently obtained in the sub-critical case $p<p_\star$, in~\cite{1602}. Here by \emph{critical} we simply mean that $\nrm w{2p,\gamma}$ scales like $\nrm{\nabla w}{2,\beta}$. One has to observe that proving symmetry breaking by establishing the linear instability is a \emph{local} method, which is based on a painful but rather straightforward linearization around the special function $w_\star$.

When the minimizer in the symmetric class is stable, \emph{i.e.}, all local perturbations that break the symmetry (in our case, non-radial perturbations) increase the energy, the problem to decide whether the optimizer is symmetric, is much more difficult. It is obvious that, in general, one cannot conclude that the minimizer is symmetric by using a local perturbation, because the minimizer in the symmetric class and the actual minimizer might not be close even in any notion of distance adapted to the functional space $\mathrm H^p_{\beta,\gamma}(\R^d)$. In general it is extremely difficult to decide, assuming stability, wether the minimizer is symmetric or not. This is a \emph{global} problem and not amenable to linear methods.

One general technique for establishing symmetry of optimizers are {\it rearrangement inequalities} and the {\it moving plane method}. These methods, however, can only be applied for functionals that are in one way or another related to the isoperimetric problem. Outside this context there are no general techniques available for understanding the symmetry of minimizers. This is quite obvious when the weights and the nonlinearity do not cooperate to decrease the energy under symmetrization and in these cases moving planes and related comparison techniques fail. As usual in nonlinear analysis, advances have always been made by studying relevant and non-trivial examples, such as finding the sharp constant in Sobolev's inequality~\cite{Aubin-76,Talenti-76}, the Hardy-Littlewood-Sobolev inequality~\cite{Lieb-83} or the logarithmic Sobolev inequality~\cite{Gross75}, to mention classical examples. In all these cases symmetrization and the moving plane methods work. Likewise, these techniques can be applied, in the case of Caffarelli-Kohn-Nirenberg inequalities, to prove that symmetry holds if $p=p_\star$ and $\beta>0$. In fact using symmetrization methods, fairly good ranges have been achieved in~\cite{MR1734159}. The results, however, are not optimal and can be improved by direct energy and spectral estimates as in~\cite{DEL2011}. Various perturbation techniques have also been implemented, as in~\cite{DELT09,DMN2015}, to extend the region of the parameters for which symmetry is known but the method, at least in~\cite{DELT09} and related papers, is not constructive. To establish the optimal symmetry range in~\eqref{CKN}, and thus determine the sharp constant in the Caffarelli-Kohn-Nirenberg inequalities, a new method had to be designed. What has been proved in~\cite{DEL2015} in the critical case $p=p_\star$, and extended in~\cite{DELM} to the sub-critical case $1<p<p_\star$, is that the symmetry breaking range given in~\eqref{set-symm-breaking} is optimal, \emph{i.e.}, symmetry holds in the region of admissible parameters that is complementary to the region in which symmetry breaking was established.

The strategy used in~\cite{DEL2015,DELM} to prove symmetry in the desired parameter range consists of perturbing the functional about the (unknown) critical point in a particular direction. Notwithstanding what has been said before about perturbations being local, the direction depends in a non-linear fashion on the critical point. It turns out that this perturbation vanishes precisely if the critical point is a radial optimizer. Of course, this begs the question how this direction can be found. In the case at hand it turns out that the functional is monotone under the action of a particular non-linear flow, and the derivative of the functional at a critical point turns out to be strictly negative unless the critical point is a radial optimizer. In carrying out this program one has to perform integration by parts and a good deal of work enters in proving the necessary regularity properties of the critical points that justify these computations.

A more appealing possibility is to use the fact that the non-linear flow, written in suitable variables, converges to a Barenblatt profile. Starting with any reasonable initial condition one would, as above, differentiate the functional along the flow and, in a formal fashion, see that that the functional decreases as time tends to infinity towards its minimal value. In addition to having an intuitive approach, one would potentially obtain correction terms to the inequality. This can be carried out, but so far the corresponding computations are formal because they rely on various integrations by parts that have to be justified. It is the first purpose of this paper to (partially) fill this gap and establish the optimal symmetry range using the full picture of entropy methods, at least as far as integration by parts on unbounded domains is concerned. In the case of non constant coefficients, the problems that might arise when dealing with the singularities of the weights at $x=0$ poses additional difficulties which are not studied in this paper, so that our results are still formal in the weighted case. But at least we make what we think is a significant step towards a complete parabolic proof.

Additionally a method based on a parabolic flow provides for free an integral remainder term, and sheds a fresh light on the method used in~\cite{DEL2015,DELM}. The results in~\cite{DEL2015,DELM} are surprising in the sense that the locally stable radial optimizers are precisely the global optimizers. From the flow perspective, however, this can be understood, because stationarity under the flow characterizes all critical points. The flow monotonously decreases the functional associated with~\eqref{CKN}: this also explains why we are able to extend a local property (the linear stability of radial solutions) to a global stability result (the uniqueness, up to the invariances, of the critical point).

The parabolic approach is based on an inequality between the \emph{Fisher information} and its time derivative, \emph{i.e.}, the \emph{production of Fisher information}, and provides us with the optimal range of symmetry. This is a remarkable fact, common to various nonlinear diffusion equations, that can be explained as follows. When there are no weights, the optimality in the entropy -- entropy production inequality is achieved through a linearization which also provides us with large-time asymptotic rates of convergence. As a consequence the best constant in the inequality is equal to the optimal constant which arises from the computation of Fisher information -- production of Fisher information~(see~\cite{MR3155209,juengel2016entropy}), and which is also reached in the large-time asymptotics. With general weights, the picture is actually slightly more complex, as discussed in~\cite{1602,1601}, but by studying the large time asymptotics, one can at least understand why the optimal symmetry range is achieved in our flow approach. This will be detailed in the last section of this paper. One more comment has to be done at this point. Quite generally computations based on the second derivative of an entropy with respect to the time, along a flow, are known as Bakry-Emery or \emph{carr\'e du champ} methods. The geometry or the presence of an external potential usually allows us to relate a positivity estimate of the curvature or a uniform convexity bound of the potential with a rate of decay of the Fisher information. In the \emph{R\'enyi entropy powers} approach, as can be seen from~\cite{1501}, there is no such bound neither on the curvature nor on the potential: what matters is only the fact that we apply a nonlinear flow to some nonlinear quantities. The interplay of the various quantities that are generated by integrations by parts and hitting powers of functions when taking derivatives delivers nontrivial coefficients that allows to relate the Fisher information with its time derivative, \emph{i.e.}, the production of Fisher information. As explained below, in order to control the boundary terms, we are dealing with the more classical setting of \emph{relative entropies} and \emph{self-similar variables}. By making the link with R\'enyi entropy powers, we finally get rid of any geometry or convexity requirements on an external potential. Although this is a side remark of our paper, we believe that it is of interest by itself.

In~\cite{DEL2015,DELM}, we analyzed the symmetry properties not only of the extremal functions of~\eqref{CKN}, but also of all positive solutions in $\mathrm H^p_{\beta,\gamma}(\R^d)$ of the corresponding Euler-Lagrange equations, \emph{i.e.}, up to a multiplication by a constant and a dilation, of
\be{ELeq}
-\,\mbox{div}\,\big(|x|^{-\beta}\,\nabla w\big)=|x|^{-\gamma}\,\big(w^{2p-1}-\,w^p\big)\quad\mbox{in}\quad\R^d\setminus\{0\}\,.
\ee
\begin{theorem}\label{Thm:Rigidity}~\cite{DEL2015,DELM} {\sl Under Condition~\eqref{parameters} assume that
\be{Symmetry condition}
\mbox{either}\quad\beta\le\beta_{\rm FS}(\gamma)\quad\forall\,\gamma<0\,,\quad\mbox{or}\quad\gamma\ge0\,.
\ee
Assume that $d\ge2$ and $(\beta,\gamma)\neq(0,0)$. Then all positive solutions of~\eqref{ELeq} in $\mathrm H^p_{\beta,\gamma}(\R^d)$ are radially symmetric and, up to a scaling and a multiplication by a constant, equal to $w_\star$.}\end{theorem}
Theorem~\ref{Thm:Rigidity} determines the optimal symmetry range, as shown by~\eqref{set-symm-breaking}. Our first result is actually a more precise version of Theorem~\ref{Thm:Rigidity}, under a regularity assumption at $x=0$ that still has to be proved.

\subsection{Main result}\label{Sec:Main}

The \emph{R\'enyi entropy power} functional relates~\eqref{poro} with~\eqref{GN}. We adopt a similar approach in the weighted case. Let us consider the derivative of the \emph{generalized R\'enyi entropy power} functional, defined up to a multiplicative constant as
\[
\mathsf G[v]:=\(\ird{v^m\,|x|^{-\gamma}}\)^{\sigma-1}\ird{v\,|\nabla\P|^2\,|x|^{-\beta}}
\]
where $\P:=\frac m{1-m}\,v^{m-1}$ is the \emph{pressure variable}, $\sigma:=\frac{m-m_c}{1-m}$ and $m_c:=\frac{d-2-\beta}{d-\gamma}$. The exponent $m$ is the one which appears in~\eqref{FD}, and it is such that $m\in[m_1,1)$. With
\[
n:=2\,\frac{d-\gamma}{\beta+2-\gamma}\,,
\]
we observe that $m_1=1-1/n$, so that, when $(\beta,\gamma)=(0,0)$, we have $n=d$ and $m_1=1-1/d$. As we will see later using scalings, $n$ plays the role of a \emph{dimension}. With $\mu_\star:=(d-\gamma)\,(m-m_c)$, $\mu:=n\,(m-m_c)$, $\kappa:=\big(\frac{2\,m}{1-m}\big)^{1/\mu}$, the function
\be{SelfSim}
v_\star(t,x):=\frac1{\kappa^n(\mu\,t)^{n/\mu}}\,\mathcal B_\star\!\(\frac x{\kappa^{\mu/\mu_\star}\,(\mu\,t)^{1/\mu_\star}}\)\quad\mbox{where}\quad\mathcal B_\star(x):=\(1+|x|^{2+\beta-\gamma}\)^{-1/(1-m)}\quad\forall\,x\in\R^d\,,\forall\,t>0
\ee
is a self-similar solution of~\eqref{FD}. Here $\mathcal B_\star$ is the Barenblatt profile with mass $M_\star:=\ird{\mathcal B_\star\,|x|^{-\gamma}}=\nrm{\mathcal B_\star}{1,\gamma}$. A simple computation shows that $\mathsf G[v_\star]$ does not depend on $t>0$ and $\nrm{v_\star}{1,\gamma}=M_\star$. Theorem~\ref{Thm:Rigidity} is equivalent to prove that
\[
\mathsf G[v]\ge\mathsf G[v_\star]
\]
for any nonnegative function $v$ such that $\nrm v{1,\gamma}=\nrm{v_\star}{1,\gamma}=M_\star$, if~\eqref{parameters} and~\eqref{Symmetry condition} hold.

In~\cite{DEL2015,DELM} we proved Theorem~\ref{Thm:Rigidity} using elliptic methods and well chosen multipliers inspired by the heuristics arising from the parabolic equation~\eqref{FD}. However, so far, we were not able to deal with the time-depen\-dent solution by lack of estimates for justifying integrations by parts and this is why we only worked with the elliptic equation. In this paper we study the evolution problem. When $(\beta,\gamma)\neq(0,0)$, we are not yet able to deal with the possible singularities of the solutions to~\eqref{FD} at the origin, but otherwise we can handle all integrations by parts. The method is based on the approximation of the solution on a ball in self-similar variables, with no-flux boundary conditions on the boundary, and then by letting the radius of the ball go to infinity. By using parabolic methods to prove Theorem~\ref{Thm:Rigidity}, we obtain improvements of~\eqref{GN} and~\eqref{CKN}, with a remainder term computed as an integral term along the flow.

Let us define
\be{hmu}
\mathsf h(t):=\(1+\frac{2\,m}{1-m}\,\mu\,t\)^{1/\mu}\quad\forall\,t\ge0\,,\quad\mbox{with}\quad\mu=2\,\frac{2+\beta-d+m\,(d-\gamma)}{2+\beta-\gamma}\,.
\ee
\begin{theorem}\label{Thm:Monotonicitylight} Let $d\ge2$. Under Condition~\eqref{parameters}, if
\[
\mbox{either}\quad\beta<\beta_{\rm FS}(\gamma)\quad\forall\,\gamma\le0\,,\quad\mbox{or}\quad\gamma>0\,,
\]
then there exists a positive constant $\mathcal C$ depending only on $\beta$, $\gamma$ and $d$ such that the following property holds. If~$v_0$ satisfies $\nrm{v_0}{1,\gamma}=M_\star$ and if there exist two positive constants $C_1$ and $C_2$ such that
\be{entredeuxBarenblatt}
\(C_1+|x|^{2+\beta-\gamma}\)^{-1/(1-m)}\le v_0(x)\le\(C_2+|x|^{2+\beta-\gamma}\)^{-1/(1-m)}\quad\forall\,x\in\R^d\,,
\ee
then for any positive solution of~\eqref{FD} with initial datum $v_0$ we have
\[
\mathsf G[v(t,\cdot)]\ge\mathsf G[v_\star]+\,\mathcal C\int_t^\infty\mathsf h(s)^{3\,\mu-2}\ird{v^m(s,x)\,\frac{|\nabla_\omega v^{m-1}(s,x)|^2}{|x|^4}\,|x|^{\gamma-2\beta}}\;ds\quad\forall\,t\ge0\,,
\]
if $v$ is \emph{smooth at $x=0$} for any $t\ge0$. Here $\mathsf h$ and $\mu$ are defined by~\eqref{hmu}.\end{theorem}
Here $\omega=\frac x{|x|}$, and $\nabla_\omega$ denotes the gradient with respect to angular derivatives. The explicit expression of~$\mathcal C$ and further remainder terms will be given in Theorem~\ref{Thm:Monotonicity}, in Section~\ref{Sec:DirectWeigthed}. The condition that $v$ is \emph{smooth at $x=0$} means that integrations by parts can be done without paying attention to the weight in a neighborhood of $x=0$. Condition~\eqref{entredeuxBarenblatt} may seem very restrictive, but it is probably not since, as explained in the introduction of~\cite{1602}, it is expected that for any smooth initial datum $v_0$ with finite mass, condition~\eqref{entredeuxBarenblatt} will be satisfied by any solution after some finite time~$t>0$. At least this is known from~\cite{MR2261689} when $(\beta,\gamma)=(0,0)$.

The mass normalization $\nrm{v_0}{1,\gamma}=M_\star$ simplifies the computations but the result can easily be generalized to any positive mass. The smoothness condition at $x=0$ simply means that all computations can be carried in $\R^d\setminus B_\varepsilon$ where $B_\varepsilon$ is the ball of radius $\varepsilon>0$ centered at the origin, and that the boundary terms on $\partial B_\varepsilon$ vanish as $\varepsilon\to0$. Up to the smoothness assumption, the result of Theorem~\ref{Thm:Monotonicitylight} is stronger than the result of Theorem~\ref{Thm:Rigidity}. Indeed, if $m$ and $p$ are related by $p=\frac1{2\,m-1}$ and if $w$ solves~\eqref{ELeq}, then we have that
\[
\frac d{dt}\mathsf G[v(t,\cdot)]=0
\]
at $t=0$, where $v$ is the solution of~\eqref{FD} with initial datum $v_0=w^{2p}$. This is enough to conclude that $\mathsf R[v_0]=0$ which, as shown in~\cite{DELM}, implies the result of Theorem~\ref{Thm:Rigidity}.

\subsection{Outline of the paper}
Our goals are:
\begin{enumerate}
\item To give a proof of the monotonicity of $\mathsf G$ for the solution to the evolution equation and establish the remainder term of Theorem~\ref{Thm:Monotonicity} under the smoothness assumption of the solutions of~\eqref{FD} at $x=0$.
\item To study the outer boundary terms by using self-similar variables and an approximation scheme on large balls. The novelty here is that we are able to justify the integrations by parts away from the origin for the solution to the evolution problem~\eqref{FD}.
\item To study the role of large time asymptotics and of the linearized problem, and consequently explain why the method in~\cite{DEL2015,DELM} determines the optimal range for symmetry breaking. Corresponding results are stated in Section~\ref{Sec:Lin}.
\end{enumerate}

Weights induce various technicalities, so that, in order to emphasize the strategy, we also consider the case without weights. In that case Theorem~\ref{Thm:Monotonicity} is rigorous without any smoothness assumption on the solution at $x=0$. This is not by itself new, but at least two observations are new:\\
$\bullet$ The equivalence of the R\'enyi entropy powers introduced by G.~Savar\'e and G.~Toscani in~\cite{MR3200617} and the computation based on the relative Fisher information in self-similar variables,\\
$\bullet$ The characterization of the optimality case in the \emph{Fisher information -- production of Fisher information} inequality, which explains why computations based on flows provide us with the optimal constant in the entropy -- entropy production inequality and, as a consequence, in the Gagliardo-Nirenberg inequality~\eqref{GN}, when there are no weights.

\section{Gagliardo-Nirenberg inequalities, fast diffusion and boundary terms}

To start with and clarify our strategy, let us consider the case of the Gagliardo-Nirenberg inequality~\eqref{GN}, without weights, \emph{i.e.}, $(\beta,\gamma)=(0,0)$ on the Euclidean space $\R^d$. In this section, integrations by parts will be fully justified. We recall that, when $(\beta,\gamma)=(0,0)$, $q=1/(2\,m-1)$ is in the range $1<q\le d/(d-2)$, which means that $m\in[m_1,1)$ with $m_1=1-1/d$ and $n=d$. Here we implicitly assume that $d\ge3$. The cases $d=1$ and $d=2$ can also be covered, with $q\in(1,+\infty)$, and the additional restriction that $m>0$ if $d=1$.

\subsection{A proof of Gagliardo-Nirenberg inequalities based on the R\'enyi entropy powers}\label{Sec:GN}

\subsubsection{Variation of the Fisher information along the flow}\label{Sec:VarFisher}

Let $v$ be a smooth function on $\R^d$ and define the \emph{Fisher information} as
\[
\mathsf I[v]:=\ird{v\,|\nabla\P|^2}\quad\mbox{with}\quad\P=\frac m{1-m}\,v^{m-1}\,.
\]
Here $\P$ is the \emph{pressure variable}. If $v$ solves~\eqref{poro}, in order to compute $\mathsf I':=\frac d{dt}\mathsf I[v(t,\cdot)]$, we will use the fact that
\be{Eqn:p}
\frac{\partial\P}{\partial t}=(1-m)\,\P\,\Delta\P-|\nabla\P|^2\,.
\ee
Using~\eqref{poro} and~\eqref{Eqn:p}, we can compute
\begin{multline*}
\mathsf I'=\frac d{dt}\ird{v\,|\nabla\P|^2}=\ird{\frac{\partial v}{\partial t}\,|\nabla\P|^2}+\,2\ird{v\,\nabla\P\cdot\nabla\frac{\partial\P}{\partial t}}\\
=\ird{\Delta(v^m)\,|\nabla\P|^2}+\,2\ird{v\,\nabla\P\cdot\nabla\Big((m-1)\,\P\,\Delta\P+|\nabla\P|^2\Big)}\,.
\end{multline*}
The key computation relies on integrations by parts and requires a sufficient decay of the solutions as $|x|\to+\infty$ to ensure that all integrals are finite, including the boundary integrals. In the next result, we focus on the algebra of the integrations by parts used to deal with the r.h.s.~of the above equality and time plays no role. How to apply this computation to a solution of the parabolic problem will be explained afterwards.
\begin{lemma}\label{Lem:DerivFisher} Assume that $v$ is a smooth and rapidly decaying function on $\R^d$, as well as its derivatives. If we let $\P:=\frac m{1-m}\,v^{m-1}$, then we have
\be{BLW}
\ird{\Delta(v^m)\,|\nabla\P|^2}+\,2\ird{v\,\nabla\P\cdot\nabla\Big((m-1)\,\P\,\Delta\P+|\nabla\P|^2\Big)}=-\,2\ird{v^m\,\Big(\|\mathrm D^2\P\|^2-\,(1-m)\,(\Delta\P)^2\Big)}\,.
\ee
\end{lemma}
\begin{proof} We follow the computation of~\cite{1501} or~\cite[Appendix~B]{MR3200617}.
\begin{eqnarray*}
&&\hspace*{-18pt}\ird{\Delta(v^m)\,|\nabla\P|^2}+\,2\ird{v\,\nabla\P\cdot\nabla\Big((1-m)\,\P\,\Delta\P-|\nabla\P|^2\Big)}\\
&=&\ird{v^m\,\Delta\,|\nabla\P|^2}+\,2\,(1-m)\ird{v\,\P\,\nabla\P\cdot\nabla\Delta\P}+\,2\,(1-m)\ird{v\,|\nabla\P|^2\,\Delta\P}-\,2\ird{v\,\nabla\P\cdot\nabla\,|\nabla\P|^2}\\
&=&-\ird{v^m\,\Delta\,|\nabla\P|^2}+\,2\,(1-m)\ird{v\,\P\,\nabla\P\cdot\nabla\Delta\P}+\,2\,(1-m)\ird{v\,|\nabla\P|^2\,\Delta\P}
\end{eqnarray*}
where the last line is given by the observation that $v\,\nabla\P=-\,\nabla(v^m)$ and an integration by parts:
\[
-\ird{v\,\nabla\P\cdot\nabla\,|\nabla\P|^2}=\ird{\nabla(v^m)\cdot\nabla\,|\nabla\P|^2}=-\ird{v^m\,\Delta\,|\nabla\P|^2}\,.
\]
1) Using the elementary identity
\[
\frac12\,\Delta\,|\nabla\P|^2=\|\mathrm D^2\P\|^2+\nabla\P\cdot\nabla\Delta\P\,,
\]
we get that
\[
\ird{v^m\,\Delta\,|\nabla\P|^2}=2\ird{v^m\,\|\mathrm D^2\P\|^2}+2\ird{v^m\,\nabla\P\cdot\nabla\Delta\P}\,.
\]
2) Since $v\,\nabla\P=-\,\nabla(v^m)$, an integration by parts gives
\[
\ird{v\,|\nabla\P|^2\,\Delta\P}=-\ird{\nabla(v^m)\cdot\nabla\P\,\Delta\P}=\ird{v^m\,(\Delta\P)^2}+\ird{v^m\,\nabla\P\cdot\nabla\Delta\P}
\]
and with $v\,\P=\frac m{1-m}\,v^m$ we find that
\[
2\,(1-m)\ird{v\,\P\,\nabla\P\cdot\nabla\Delta\P}+\,2\,(1-m)\ird{v\,|\nabla\P|^2\,\Delta\P}
=2\,(1-m)\ird{v^m\,(\Delta\P)^2}+2\ird{v^m\,\nabla\P\cdot\nabla\Delta\P}\,.
\]
Collecting terms establishes~\eqref{BLW}.
\end{proof}
The result of Lemma~\ref{Lem:DerivFisher} can be applied to a solution of~\eqref{poro}.
\begin{corollary}\label{Cor:DerivFisher} If $v$ solves~\eqref{poro} with initial datum $v(x,t=0)=v_0(x)\ge0$ such that $\ird{v_0}=M_\star$, $\ird{v_0^m}<+\infty$ and $\ird{|x|^2\,v_0}<+\infty$, then $\P=\frac m{1-m}\,v^{m-1}$ solves~\eqref{Eqn:p} and
\[
\mathsf I'=-\,2\ird{v^m\,\Big(\|\mathrm D^2\P\|^2-\,(1-m)\,(\Delta\P)^2\Big)}\,.
\]
\end{corollary}
\begin{proof} If we perform the same computations as in the proof of Lemma~\ref{Lem:DerivFisher} in a ball $B_R$ (instead of $\R^2$), we find additional boundary terms
\begin{multline*}\label{BLW--bdryterms}
\mathsf I'=\frac d{dt}\iBR{v\,|\nabla\P|^2}=-\,2\iBR{v^m\,\Big(\|\mathrm D^2\P\|^2+(m-1)\,(\Delta\P)^2\Big)}\\
+\idBR{\omega\cdot\Big(\nabla(v^m)\,|\nabla\P|^2+v^m\,\nabla\,|\nabla\P|^2+2\,(m-1)\,v^m\,\nabla\P\,\Delta\P\Big)}\,.
\end{multline*}
Here $d\sigma$ denotes the measure induced on $\partial B_R$ by Lebesgue's measure. It follows from~\cite[Theorem~2, (iii)]{BBDGV} that these boundary terms vanish as $R\to+\infty$, which completes the proof.
\end{proof}

\subsubsection{Concavity of the R\'enyi entropy powers and consequences}\label{Sec:Concavity}

Lemma~\ref{Lem:DerivFisher} establishes that $\mathsf I'\le0$ if $m\ge m_1$ with $m_1=1-1/d$. Indeed we have the identity
\[\label{BLWid}
\|\mathrm D^2\P\|^2=\tfrac1d\,(\Delta\P)^2+\left\|\,\mathrm D^2\P-\tfrac1d\,\Delta\P\,\mathrm{Id}\,\right\|^2
\]
and, as a consequence, we obtain
\[
\mathsf I'=-\,2\ird{{v^m\,\Big(\left\|\,\mathrm D^2\P-\tfrac1d\,\Delta\P\,\mathrm{Id}\,\right\|^2}}-\,2\,(m-m_1)\ird{{v^m\,(\Delta\P)^2}}\,.
\]
In the sub-critical range $m_1<m<1$, let us define the \emph{entropy} as $\mathsf E=\ird{v^m}$ and observe that, if $v$ solves~\eqref{poro},
\[
\mathsf E'=(1-m)\,\mathsf I\,.
\]
Next we introduce the \emph{R\'enyi entropy power} given by $\mathsf F=\mathsf E^\sigma$ with
\[
\sigma:=\frac 2d\,\frac1{1-m}-1\,.
\]
Using Lemma~\ref{Lem:DerivFisher}, we find that $\mathsf F''=\(\mathsf E^\sigma\)''$ can be computed as
\begin{multline*}
\frac1{\sigma\,(1-m)}\,\mathsf E^{2-\sigma}\,\mathsf F''=(1-m)\,(\sigma-1)\(\ird{v\,|\nabla\P|^2}\)^2\\
\hspace*{3cm}-\,2\(\frac1d+m-1\)\ird{v^m}\ird{v^m\,(\Delta\P)^2}-\,2\ird{v^m}\ird{v^m\,\left\|\,\mathrm D^2\P-\tfrac1d\,\Delta\P\,\mathrm{Id}\,\right\|^2}\,.
\end{multline*}
Using $v\,\nabla\P=-\,\nabla(v^m)$, we know that
\[
\ird{v\,|\nabla\P|^2}=-\ird{\nabla(v^m)\cdot\nabla\P}=\ird{v^m\,\Delta\P}\,,
\]
which implies that
\[
\ird{v^m\,\left|\Delta\P-\frac{\(\ird{v\,|\nabla\P|^2}\)^2}{\ird{v^m}}\right|^2}=\ird{v^m\,|\Delta\P|^2}-\frac{\(\ird{v\,|\nabla\P|^2}\)^2}{\ird{v^m}}\,.
\]
Hence we get that $\mathsf F''=-\,\sigma\,(1-m)\,\mathsf R$, where
\be{R}
\mathsf R[v]:=(\sigma-1)\,(1-m)\,\mathsf E[v]^{\sigma-1}\ird{v^m\,\left|\Delta\P-\frac{\ird{v\,|\nabla\P|^2}}{\ird{v^m}}\right|^2}+\,2\,\mathsf E[v]^{\sigma-1}\ird{v^m\,\left\|\,\mathrm D^2\P-\tfrac1d\,\Delta\P\,\mathrm{Id}\,\right\|^2}\,.
\ee
This proves that $\sigma\,(1-m)\,\mathsf G=\mathsf F'$ is nonincreasing (so that the function $t\mapsto\mathsf F(t)$ is concave).

\subsubsection{Large time asymptotics and consequences}\label{Sec:LargeTime}
The large time behavior of the solution of~\eqref{poro} is governed by the source-type \emph{Barenblatt solutions}
\[
v_\star(t,x):=\frac1{\kappa^d(\mu\,t)^{d/\mu}}\,\mathcal B_\star\!\(\frac x{\kappa\,(\mu\,t)^{1/\mu}}\)\quad\mbox{where}\quad\mu:=2+d\,(m-1)=d\,(m-m_c)\,,\quad\kappa:=\Big(\frac{2\,m}{1-m}\Big)^{1/\mu}\,,
\]
where $\mathcal B_\star$ is the Barenblatt profile
\[
\mathcal B_\star(x):=\big(1+|x|^2\big)^{1/(m-1)}
\]
with mass $M_\star:=\ird{\mathcal B_\star}$. We recall that $(\beta,\gamma)=(0,0)$ and, as a consequence, $n=d$ and $\mu=\mu_\star$: notations are consistent with those of~\eqref{SelfSim}. To obtain the expression of $v_\star$, it is standard to rephrase the evolution equation~\eqref{poro} in self-similar variables as follows. If we consider a solution $v$ of~\eqref{poro} and make the change of variables
\be{TDRS}
v(t,x)=\frac1{\kappa^d\,R^d}\,u\!\(\tau,\frac x{\kappa\,R}\)\quad\mbox{where}\quad\frac{dR}{dt}=R^{1-\mu}\,,\quad R(0)=R_0=\kappa^{-1}\quad\mbox{and}\quad\tau(t):=\tfrac12\,\log\(\frac{R(t)}{R_0}\)\,,
\ee
then the function $u$ solves
\be{RescaledPoro}
\frac{\partial u}{\partial\tau}+\nabla\cdot\Big[u\(\nabla u^{m-1}-\,2\,x\)\Big]=0\,,\quad(\tau,x)\in\R^+\times\R^d\,.
\ee
It is straightforward to check that $\mathcal B_\star$ is a stationary solution of~\eqref{RescaledPoro} and it is well known that $\mathcal B_\star$ attracts all nonnegative solutions with mass $M_\star$ at least if $m\in(m_c,1)$. Since
\[
R(t)=\(R_0^\mu+\mu\,t\)^{1/\mu}=(\mu\,t)^{1/\mu}\,\big(1+o(1)\big)\quad\mbox{as}\quad t\to+\infty\,,
\]
this means that
\[
v(t,x)\sim v_\star(t,x)\quad\mbox{as}\quad t\to+\infty\,.
\]
We refer to~\cite{BBDGV} for details and further references. As a consequence, we obtain that
\[
\lim_{t\to\infty}\mathsf G[v(t,\cdot)]=\lim_{t\to\infty}\mathsf G[v_\star(t,\cdot)]=\mathsf G[\mathcal B_\star]\,,
\]
because
\[
\mathsf G[v]=\(\ird{v^m}\)^{\sigma-1}\ird{v\,|\nabla\P|^2}
\]
defined in Section~\ref{Sec:Main} is scale invariant. The fact that $\mathsf G[v(t,\cdot)]$ is a nonincreasing function of $t$ means that
\[
\mathsf G[v_0]\ge\mathsf G[v(t,\cdot)]\ge\mathsf G[\mathcal B_\star]
\]
for any $t\ge0$, which is exactly equivalent to~\eqref{GN} as noted in~\cite{1501}. By keeping track of the remainder term, we get an improved inequality.
\begin{proposition} Under the assumptions of Corollary~\ref{Cor:DerivFisher}, with $\mathsf R$ defined by~\eqref{R}, for all $t\ge 0$, we have
\[
\mathsf G[v_0]=\mathsf G[v_\star]+\int_0^\infty\mathsf R[v(t,\cdot)]\,dt\,.
\]
\end{proposition}
If we write $v_0^{m-1/2}=M_\star^{m-1/2}\,w/\nrm w{2q}$ with $q=1/(2\,m-1)$, then this inequality amounts to
\[
\tfrac{(2\,m-1)^2}{4\,m^2}\,\Big(\mathsf G[v_0]-\mathsf G[\mathcal B_\star]\Big)=\frac{\nrm{\nabla w}2^2\,\nrm w{q+1}^{2\,(1-\theta)/\theta}}{\nrm w{2q}^{2/\theta}}-\mathsf C_{\rm{GN}}^{2/\theta}\ge0\quad\mbox{with}\quad\mathsf C_{\rm{GN}}:=\(\tfrac{(2\,m-1)^2}{4\,m^2}\,\mathsf G[\mathcal B_\star]\)^{\theta/2}\,.
\]
In this way, we recover the Gagliardo-Nirenberg inequality that was established in~\cite{MR1940370}, with an additional remainder term.
\begin{corollary}\label{Cor:GN} If $1<q\le\frac d{d-2}$ and $d\ge3$, or $q>1$ and $d=1$ or $d=2$, then~\eqref{GN} holds with optimal constant~$\mathsf C_{\rm{GN}}$ as above. Moreover, equality holds in~\eqref{GN} if and only if $w^{2q}=\mathcal B_\star$, up to translations, multiplication by constants and scalings. With the above notations, one has the improved inequality
\[
\nrm{\nabla w}2^2\,\nrm w{q+1}^{2\,(1-\theta)/\theta}-\mathsf C_{\rm{GN}}^{2/\theta}\,\nrm w{2q}^{2/\theta}=\tfrac{(2\,m-1)^2}{4\,m^2}\,\nrm w{2q}^{2/\theta}\,\int_0^\infty\mathsf R[v(t,\cdot)]\,dt\quad\forall\,w\in H^p_{0,0}(\R^d)
\]
if $v$ solves~\eqref{poro} with initial datum $v_0$ such that $v_0^{m-1/2}=M_\star^{m-1/2}\,w/\nrm w{2q}$ and $q=1/(2\,m-1)$.\end{corollary}
\begin{proof} The only point that deserves a discussion is the equality case. Solving simultaneously
\[
\Delta\P-\frac{\ird{v\,|\nabla\P|^2}}{\ird{v^m}}=0\quad\mbox{and}\quad\mathrm D^2\P-\tfrac1d\,\Delta\P\,\mathrm{Id}=0
\]
shows that $\P(x)=\mathsf a+\mathsf b\,|x-x_0|^2$ for some real constants $\mathsf a$ and $\mathsf b$, and for some $x_0\in\R^d$.
\end{proof}

\subsection{The entropy -- entropy production method in rescaled variables}\label{Sec:RescaledVariables}

Here we follow the computations of~\cite[Section~2]{MR3103175} (also see~\cite[Proof~of~Theorem~2.4, pp.~33-36]{juengel2016entropy}) and emphasize the role of the boundary terms when the problem is restricted to a ball. The major advantage of self-similar variables is that we control the sign of these boundary terms. Such computations can be traced back to~\cite{MR1777035,MR1853037,MR1986060} and are directly inspired by the \emph{carr\'e du champ} or Bakry-Emery method introduced in~\cite{Bakry-Emery85}. The algebra is slightly more involved than the one of Section~\ref{Sec:GN} because of the presence of a drift term. The main advantage of this framework is that boundary terms have a definite sign, which is important in preparation of the computations of Section~\ref{Sec:DirectWeigthed}, in the weighted case.

For a while we will consider Eq.~\eqref{RescaledPoro} written on ball $B_R$ instead of $\R^d$, with \emph{no-flux boundary condition}. Let $u=u(\tau,x)$ be a solution of
\be{Eqn2}
\frac{\partial u}{\partial\tau}+\nabla\cdot\Big[u\(\nabla u^{m-1}-\,2\,x\)\Big]=0\quad\tau>0\,,\quad x\in B_R
\ee
where $B_R$ is a centered ball in $\R^d$ with radius $R>0$, and assume that $u$ satisfies no-flux boundary conditions
\[
\(\nabla u^{m-1}-\,2\,x\)\cdot\omega=0\quad\tau>0\,,\quad x\in\partial B_R\,.
\]
On $\partial B_R$, $\omega=x/|x|$ denotes the unit outgoing normal vector to $\partial B_R$. We define
\[
z(\tau,x):=\nabla u^{m-1}-\,2\,x
\]
so that Eq.~\eqref{Eqn2} can be rewritten with its boundary conditions as
\[
\frac{\partial u}{\partial\tau}+\nabla\cdot\(u\,z\)=0\quad\mbox{in}\quad B_R\,,\quad z\cdot\omega=0\quad\mbox{on}\quad\partial B_R\,.
\]
We recall that $m\in[m_1,1)$ where $m_1=1-1/d$. It is straightforward to check that
\[
\frac{\partial z}{\partial\tau}=(1-m)\,\nabla\big(u^{m-2}\,\nabla\cdot\(u\,z\)\big)\,.
\]
With these definitions, the time-derivative of \emph{relative Fisher information}
\[
\mathcal I_R[u]:=\iBR{u\,|z|^2}=\iBR{u\,\left|\nabla u^{m-1}-\,2\,x\right|^2}
\]
can be computed as
\begin{multline*}
\frac d{d\tau}\iBR{u\,|z|^2}=\iBR{\frac{\partial u}{\partial\tau}\,|z|^2}+2\iBR{u\,z\cdot\frac{\partial z}{\partial\tau}}\\
=\iBR{u\,z\,\cdot\nabla\,|z|^2}-\,2\iBR{u\,z\cdot\nabla\(z\cdot\nabla u^{m-1}+(m-1)\,u^{m-1}\,\nabla\cdot z\)}
\end{multline*}
using the above equations. By definition of $z$, we have
\begin{multline*}
\frac d{d\tau}\iBR{u\,|z|^2}=\iBR{u\,z\,\cdot\nabla\,|z|^2}-\,2\iBR{u\,z\cdot\nabla\(|z|^2+2\,z\cdot x+(m-1)\,u^{m-1}\,\nabla\cdot z\)}\\
=-\iBR{u\,z\,\cdot\nabla\,|z|^2}-\,2\iBR{u\,z\cdot\nabla\(2\,z\cdot x+(m-1)\,u^{m-1}\,\nabla\cdot z\)}\\
=-\iBR{\(\tfrac{m-1}m\,\nabla u^m-\,2\,x\,u\)\,\cdot\nabla\,|z|^2}-\,2\iBR{u\,z\cdot\nabla\(2\,z\cdot x+(m-1)\,u^{m-1}\,\nabla\cdot z\)}\,.
\end{multline*}
Let us denote by $d\sigma$ the measure induced by Lebesgue's measure on $\partial B_R$. Taking into account the boundary condition $z\cdot\omega=0$ on $\partial B_R$, we integrate by parts and get
\begin{eqnarray*}
&&\hspace*{-24pt}\frac d{d\tau}\iBR{u\,|z|^2}\\
&=&\iBR{\tfrac{m-1}m\,u^m\,\Delta\,| z|^2}+2\iBR{u\,x\,\cdot\nabla\,|z|^2}-\,4\iBR{u\,z\cdot\nabla(z\cdot x)}+\idBR{\tfrac{1-m}m\,u^m\(\omega\cdot\nabla\,|z|^2\)}\\
&&\hspace*{24pt}-\,2\,(m-1)\,\iBR{\(u\,z\cdot\nabla u^{m-1}\,(\nabla\cdot z)+u^m\,z\cdot\nabla\,(\nabla\cdot z)\)}\\
&=&\iBR{\tfrac{m-1}m\,u^m\,\Delta\,| z|^2}+2\iBR{u\,x\,\cdot\nabla\,|z|^2}-\,4\iBR{u\,z\cdot\nabla(z\cdot x)}+\idBR{\tfrac{1-m}m\,u^m\(\omega\cdot\nabla\,|z|^2\)}\\
&&\hspace*{24pt}-\,2\,(m-1)\,\iBR{\(-\tfrac{m-1}m\,u^m\(\nabla\cdot z\)^2+\tfrac1m\,u^m\,z\cdot\nabla(\nabla\cdot z)\)}\,.
\end{eqnarray*}
Using the elementary identity
\[
\frac12\,\Delta\,|\nabla\mathsf q|^2=\left\|\mathrm D^2\mathsf q\right\|^2+\nabla\mathsf q\cdot\nabla\Delta\mathsf q\,,
\]
with $\mathsf q:=u^{m-1}-1-|x|^2$ so that $z=\nabla\mathsf q$, we get that
\[
\iBR{u^m\,\Delta\,|z|^2}=2\iBR{u^m\,\left\|\mathrm D^2\mathsf q\right\|^2}+2\iBR{u^m\,z\cdot\nabla(\nabla\cdot z)}\,.
\]
Moreover, since $z\cdot\omega=0$ on $\partial B_R$, we know from~\cite[Lemma 5.2]{MR2533926},~\cite[Proposition~4.2]{MR3150642} or~\cite{MR775683} (also see~\cite{MR2435196} or~\cite[Lemma~A.3]{juengel2016entropy}) that
\[
\idBR{u^m\(\omega\cdot\nabla|z|^2\)}\le 0\,.
\]
Therefore, we have shown that
\begin{multline*}
\frac d{d\tau}\iBR{u\,|z|^2}\le2\,\frac{m-1}m\iBR{u^m\(\left\|\mathrm D^2\mathsf q\right\|^2+(m-1)\,(\Delta\mathsf q)^2\)}+2\iBR{u\,x\,\cdot\nabla\,|z|^2}-\,4\iBR{u\,z\cdot\nabla(z\cdot x)}\\
=-\,2\,\frac{1-m}m\iBR{u^m\(\left\|\mathrm D^2\mathsf q\right\|^2-\,(1-m)\,(\Delta\mathsf q)^2\)}-\,4\iBR{u\,|z|^2}
\end{multline*}
where, in the last step, we use the fact that $\frac{\partial z_j}{\partial x_i}=\frac{\partial z_i}{\partial x_j}$ to write that
\[
2\iBR{u\,x\,\cdot\nabla\,|z|^2}-\,4\iBR{u\,z\cdot\nabla(z\cdot x)}=-\,4\iBR{u\,|z|^2}\,.
\]
For any $m\in[m_1,1)$, this establishes that
\[
\frac d{d\tau}\iBR{u\,|z|^2}+4\iBR{u\,|z|^2}\le-\,2\,\frac{1-m}m\iBR{u^m\(\left\|\mathrm D^2\mathsf q\right\|^2-\,(1-m)\,(\Delta\mathsf q)^2\)}\le0\,.
\]
This allows us to prove a result similar to the one of Corollary~\ref{Cor:GN}. The \emph{relative entropy}
\[
\mathcal E_R[u]:=-\,\frac1m\iBR{\(u^m-\,\mathcal B_\star^m-\,m\,\mathcal B_\star^{m-1}\,(u-\,\mathcal B_\star)\)}
\]
is such that
\[
\frac d{d\tau}\mathcal E_r[u(\tau,\cdot)]=-\,\mathcal I_R[u(\tau,\cdot)]
\]
according to~\cite{MR1940370,BBDGV}. We deduce that
\[
\frac d{d\tau}\Big(\mathcal I_R[u(\tau,\cdot)]-\,4\,\mathcal E_R[u(\tau,\cdot)]\Big)\le-\,2\,\frac{1-m}m\iBR{u^m\(\left\|\mathrm D^2\mathsf q\right\|^2-\,(1-m)\,(\Delta\mathsf q)^2\)}\,.
\]
It turns out that for all $\tau\ge 0$,
\[
\mathcal I_R[u_0]-\,4\,\mathcal E_R[u_0]\ge\mathcal I_R[u(\tau,\cdot)]-\,4\,\mathcal E_R[u(\tau,\cdot)]\ge\mathcal I_R[\mathcal B_\star]-\,4\,\mathcal E_R[\mathcal B_\star]=0\,.
\]

For functions on $\R^d$, let us define the \emph{relative entropy}
\[
\mathcal E[u]:=-\,\frac1m\ird{\(u^m-\,\mathcal B_\star^m-\,m\,\mathcal B_\star^{m-1}\,(u-\,\mathcal B_\star)\)}\,,
\]
the \emph{relative Fisher information}
\[
\mathcal I[u]:=\ird{u\,|z|^2}=\ird{u\,\left|\nabla u^{m-1}-\,2\,x\right|^2}
\]
and
\[
\mathcal R[u]:=\,2\,\frac{1-m}m\ird{u^m\(\left\|\mathrm D^2\mathsf q\right\|^2-\,(1-m)\,(\Delta\mathsf q)^2\)}\,.
\]
\begin{proposition}\label{Prop:GN2} If $ 1<q\le\frac d{d-2}$ and $d\ge3$, or $q>1$ and $d=1$ or $d=2$, then~\eqref{GN} holds with optimal constant~$\mathsf C_{\rm{GN}}$ as above. Moreover, equality holds in~\eqref{GN} if and only if $w^{2q}=\mathcal B_\star$, up to translations, multiplication by constants and scalings. With the above notations, one has the improved inequality
\[
\mathcal I[u_0]-\,4\,\mathcal E[u_0]\ge\int_0^\infty\mathcal R[u(\tau,\cdot)]\,d\tau\,,
\]
if $u$ solves~\eqref{RescaledPoro} with initial datum $u_0\in\mathrm L^1_+(\R^d)$ such that $u_0^m$ and $u_0\,|x|^2$ are integrable.\end{proposition}
\begin{proof} To prove the result, one has to approximate a solution of~\eqref{RescaledPoro} by the solution of~\eqref{Eqn2} on the centered ball $B_R$ of radius $R$, and extend it to $\R^d\setminus B_R$ by $\mathcal B_\star$. By passing to the limit as $R\to+\infty$, the result follows.\end{proof}

To conclude this subsection, let us list a few comments.
\begin{enumerate}
\item The method of entropy -- entropy production method in rescaled variables is not as accurate as the method of R\'enyi entropy powers. Boundary terms have a sign and can be dropped, but at the end we get an inequality instead of an equality. On the other hand, the method is very robust and applicable not only to large balls but also to any convex domain. This is the method that we shall extend to the case of the weighted evolution equation in Section~\ref{Sec:DirectWeigthed}.
\item If we replace $u$ by $w^{2q}$, with $2q=m-\frac12$, the inequality $\mathcal I[u]-\,4\,\mathcal E[u]\ge0$ amounts to
\[
\tfrac{4\,(m-1)^2}{(2\,m-1)^2}\nrm{\nabla w}2^2+\,\frac4m\,\big(1-d\,(1-m)\big)\,\nrm w{q+1}^{q+1}\ge\tfrac{4\,(m-1)^2}{(2\,m-1)^2}\nrm{\nabla\mathcal B_\star^{m-1/2}}2^2+\,\frac4m\,\big(1-d\,(1-m)\big)\,\nrm{\mathcal B_\star^{m-1/2}}{q+1}^{q+1}\,.
\]
This is a non scale-invariant, but optimal, form of the Gagliardo-Nirenberg inequality, as shown in~\cite{MR1940370}. The inequality written with $w$ replaced by $M_\star^{m-1/2}\,w/\nrm w{2q}$ is, after optimization under scalings, equivalent to inequality~\eqref{GN}. However $\mathcal R[u]$ is not invariant under scaling. In order to replace the improved inequality of Proposition~\ref{Prop:GN2} by an improved inequality similar to the one of Corollary~\ref{Cor:GN}, one should use delicate scaling properties involving the \emph{best matching Barenblatt} instead of~$\mathcal B_\star$. See~\cite{MR3103175,1501} for further considerations in this direction.
\item An interesting remark which is important in our computations and results is that for any function $\mathsf p$,
\[
\left\|\mathrm D^2\(\mathsf p+|x|^2\)\right\|^2-\,\tfrac1d\,\Big(\Delta\(\mathsf p+|x|^2\)\Big)^2=\left\|\mathrm D^2\mathsf p\right\|^2-\,\tfrac1d\,(\Delta\mathsf p)^2=\left\|\mathrm D^2\mathsf p-\,\tfrac1d\,\Delta\mathsf p\,\mathrm{Id}\right\|^2\,.
\]
As a consequence, the remainder terms in the entropy -- entropy production method in rescaled variables are very similar to the remainder terms in the R\'enyi entropy powers method, and the $|x|^2$ term plays essentially no role.
\end{enumerate}

\subsection{The two methods are identical}\label{Sec:Renyi=BE}

The computations of Sections~\ref{Sec:GN} and~\ref{Sec:RescaledVariables} look similar and are actually the same, if we do not take into consideration the boundary terms. Let us give some details.

\subsubsection{A computation based on the time-dependent rescaling}
If $v$ is a solution of~\eqref{poro}, then the function $u$ defined by the time-dependent rescaling~\eqref{TDRS} solves~\eqref{RescaledPoro}. With the choice $R_0=1/\kappa$, the initial data are identical
\[
u(\tau=0,\cdot)=u_0=v_0=v(t=0,\cdot)\,.
\]
As in Section~\ref{Sec:RescaledVariables}, let us define $z(x,\tau):=\nabla u^{m-1}-\,2\,x$ and consider the relative Fisher information
\[
\mathcal I[u]:=\ird{u\,|z|^2}=\ird{u\,\left|\nabla u^{m-1}-\,2\,x\right|^2}=\ird{u\,\left|\nabla u^{m-1}\right|^2}+\,4\ird{u\,|x|^2}-\,4\tfrac{1-m}m\,d\ird{u^m}\,.
\]

\noindent$\bullet$ If $m=m_1=1-\frac1d$, then $\frac{1-m}m\,d=\frac1m$ and, by undoing the time-dependent rescaling~\eqref{TDRS}, we obtain that
\[
\ird{u\,\left|\nabla u^{m-1}\right|^2}=\(\tfrac{1-m}m\)^2\ird{v\,|\nabla\P|^2}
\]
with $\P=\frac m{1-m}\,v^{m-1}$, because $\mu=1$ and so, since $\frac{dt}{d\tau}=\frac{1-m}m\,e^{2\,\tau}$, we get that
\[
\frac d{d\tau}\ird{u\,\left|\nabla u^{m-1}\right|^2}=\tfrac{1-m}m\,e^{2\,\tau}\,\frac d{dt}\ird{v\,|\nabla\P|^2}
\]
is nonpositive by Corollary~\ref{Cor:DerivFisher}. By the computations of Section~\ref{Sec:GN}, we obtain that
\begin{multline*}
\frac d{d\tau}\mathcal I[u(\tau,\cdot)]=\tfrac{1-m}m\,e^{2\,\tau}\,\frac d{dt}\ird{v\,|\nabla\P|^2}+4\,\frac d{d\tau}\ird{\(u\,|x|^2-\frac1m\,u^m\)}\\
\le4\,\frac d{d\tau}\ird{\(u\,|x|^2-\frac1m\,u^m\)}=-\,4\,\mathcal I[u(\tau,\cdot)]\,.
\end{multline*}

\noindent$\bullet$ If $m\in[m_1,1)$, we observe that
\[
\ird{u\,\left|\nabla u^{m-1}\right|^2}=\(\tfrac{1-m}m\)^2\,e^{4\,(\mu-1)\,\tau}\ird{v\,|\nabla\P|^2}
\]
and from $R(t)=R_0\,e^{2\,\tau}$, we deduce that $\frac{dt}{d\tau}=2\,R^\mu=\frac{1-m}m\,e^{2\,\mu\,\tau}$. A computation similar to the case $m=m_1$ gives
\begin{multline*}
\frac d{d\tau}\Big(\mathcal I[u(\tau,\cdot)]-\,4\,\mathcal E[u(\tau,\cdot)]\Big)\\
=\(\tfrac {1-m}m\)^2\,e^{4\,(\mu-1)\,\tau}\(\tfrac{1-m}m\,e^{2\,\mu\,\tau}\,\frac d{dt}\ird{v\,|\nabla\P|^2}+\,4\,(\mu-1)\ird{v\,|\nabla\P|^2}\)+\,\tfrac{4\,d}m\,(m-m_1)\frac d{d\tau}\ird{u^m}
\end{multline*}
with $\P=\frac m{1-m}\,v^{m-1}$. Using the fact that
\[
\frac d{d\tau}\ird{u^m}=m\ird{u\,\left|\nabla u^{m-1}\right|^2}-\,2\,d\,(1-m)\ird{u^m}
\]
on the one hand, and Corollary~\ref{Cor:DerivFisher} on the other hand, we end up with
\[\label{ineqqq}
\frac d{d\tau}\Big(\mathcal I[u(\tau,\cdot)]-\,4\,\mathcal E[u(\tau,\cdot)]\Big)=-\,\mathcal R_\star[u(\tau,\cdot)]\,,
\]
where
\begin{multline*}
\mathcal R_\star[u]:=2\,e^{4\,(\mu-1)\,\tau}\(\tfrac{1-m}m\,e^{2\,\mu\,\tau}\ird{v^m\,\Big(\left\|\mathrm D^2v^{m-1}\right\|^2-\,(1-m)\(\Delta v^{m-1}\)^2\Big)}-\,2\,(\mu-1)\ird{v\,\left|\nabla v^{m-1}\right|^2}\)\\
-\,\tfrac{4\,d}m\,(m-m_1)\(m\ird{u\,\left|\nabla u^{m-1}\right|^2}-\,2\,d\,(1-m)\ird{u^m}\)\,.
\end{multline*}
Notice that $\mathcal R_\star[u]$ does not depend on $\tau$ explicitly because, according to the time-dependent rescaling~\eqref{TDRS},
\begin{multline*}
\mathcal R_\star[u]=2\,\tfrac{1-m}m\ird{u^m\,\left\|\mathrm D^2u^{m-1}-\,\tfrac 1d\,\Delta u^{m-1}\,\mathrm{Id}\right\|^2}+\,2\,(m-m_1)\,\tfrac{1-m}m\ird{u^m\(\Delta u^{m-1}\)^2}\\
-\,8\,d\,(m-m_1)\ird{u\,\left|\nabla u^{m-1}\right|^2}+\,\tfrac{8\,d^2}m\,(m-m_1)\,(1-m)\ird{u^m}\,.
\end{multline*}
This can be rewritten as
\[
\mathcal R_\star[u]=2\,\tfrac{1-m}m\ird{u^m\,\left\|\mathrm D^2u^{m-1}-\,\tfrac 1d\,\Delta u^{m-1}\,\mathrm{Id}\right\|^2}+\,2\,(m-m_1)\,\tfrac{1-m}m\ird{u^m\,\left|\Delta u^{m-1}-\,2\,d\right|^2}\,.
\]
With these considerations, we obtain an improvement of Proposition~\ref{Prop:GN2}, which goes as follows.
\begin{proposition}\label{Prop:GN3} If $ 1<q\le\frac d{d-2}$ and $d\ge3$, or $q>1$ and $d=1$ or $d=2$, then~\eqref{GN} holds with optimal constant~$\mathsf C_{\rm{GN}}$ as above. Moreover, equality holds in~\eqref{GN} if and only if $w^{2q}=\mathcal B_\star$, up to translations, multiplication by constants and scalings. With the above notations, one has the improved inequality
\[
\mathcal I[u_0]-\,4\,\mathcal E[u_0]=\int_0^\infty\mathcal R_\star[u(\tau,\cdot)]\,d\tau\,,
\]
if $u$ solves~\eqref{RescaledPoro} with initial datum $u_0\in\mathrm L^1_+(\R^d)$ such that $u_0^m$ and $u_0\,|x|^2$ are integrable.\end{proposition}

\subsubsection{A direct computation in rescaled variables} Although this is equivalent to the computations of the previous subsection, it is instructive to redo the computation in the rescaled variables. Let us define $\mathsf p:=u^{m-1}$ and observe that it solves
\[
\frac{\partial\mathsf p}{\partial\tau}=(m-1)\,\mathsf p\,\Delta\mathsf p-|\nabla\mathsf p|^2+\,2\,x\cdot\nabla\mathsf p+2\,d\,(m-1)\,\mathsf p\,.
\]
For simplicity, we consider only the case $m=m_1$ and observe that
\begin{multline*}
\frac d{d\tau}\ird{u\,|\nabla\mathsf p|^2}=\frac{1-m}m\ird{\Delta(u^m)\,|\nabla\mathsf p|^2}+\,2\ird{u\,\nabla\mathsf p\cdot\nabla\Big((m-1)\,\mathsf p\,\Delta\mathsf p+|\nabla\mathsf p|^2\Big)}\\
+\,2\ird{\nabla\cdot(x\,u)\,|\nabla\mathsf p|^2}+\,4\ird{u\,\nabla\mathsf p\cdot\nabla\Big(x\cdot\nabla\mathsf p+d\,(m-1)\,\mathsf p\Big)}\\
=\frac{1-m}m\ird{\Delta(u^m)\,|\nabla\mathsf p|^2}+\,2\ird{u\,\nabla\mathsf p\cdot\nabla\Big((m-1)\,\mathsf p\,\Delta\mathsf p+|\nabla\mathsf p|^2\Big)}-\,4\ird{u\,|\nabla\mathsf p|^2}
\end{multline*}
because $d\,(m-1)=-1$ and $2\,\nabla\cdot(x\,u)\,|\nabla\mathsf p|^2+4\,u\,\nabla\mathsf p\cdot\nabla\(x\cdot\nabla\mathsf p\)=2\,\nabla\cdot\(x\,u\,|\nabla\mathsf p|^2\)$. If we write that $\P:=\frac m{1-m}\,\mathsf p$, then the r.h.s.~can be rewritten as
\begin{multline*}
\frac{1-m}m\ird{\Delta(u^m)\,|\nabla\mathsf p|^2}+\,2\ird{u\,\nabla\mathsf p\cdot\nabla\Big((m-1)\,\mathsf p\,\Delta\mathsf p+|\nabla\mathsf p|^2\Big)}\\
=\(\frac{1-m}m\)^3\left[\ird{\Delta(u^m)\,|\nabla\P|^2}+\,2\ird{u\,\nabla\P\cdot\nabla\Big((m-1)\,\P\,\Delta\P+|\nabla\P|^2\Big)}\right]
\end{multline*}
and we are back to the computations of Section~\ref{Sec:GN}. Using~\eqref{BLW} with $v=u$ and $\P=\frac m{1-m}\,v^{m-1}$, we obtain that
\[
\frac d{d\tau}\mathcal I[u(\tau,\cdot)]+\,4\,\mathcal I[u(\tau,\cdot)]=\frac d{d\tau}\ird{u\,|\nabla\mathsf p|^2}+\,4\ird{u\,|\nabla\mathsf p|^2}=-\,2\,\frac{1-m}m\ird{u^m\,\Big(\left\|\,\mathrm D^2\mathsf p-\tfrac1d\,\Delta\mathsf p\,\mathrm{Id}\,\right\|^2}
\]
if $u$ solves~\eqref{RescaledPoro}.

This concludes the section on Gagliardo-Nirenberg inequalities~\eqref{GN} and fast diffusion equations~\eqref{poro}. So far proofs are rigorous. From now on, we shall work with weights, that is, on Caffarelli-Kohn-Nirenberg inequalities~\eqref{CKN} and weighted parabolic equations~\eqref{FD}, and assume that integrations by parts can be carried out at $x=0$ without any precaution. Corresponding results will henceforth be formal.

\section{The case of the weighted diffusion equation}\label{Sec:DirectWeigthed}

In order to study the weighted evolution equation~\eqref{FD}, it is convenient to introduce as in~\cite{DEL2015,DELM} a change of variables which amounts to rephrase our problem in a space of higher, \emph{artificial dimension} $n\ge d$ (here $n$ is a dimension at least from the point of view of the scaling properties), or to be precise to consider a weight $|x|^{n-d}$ which is the same in all norms. With
\[\label{ctsrel}
\alpha=1+\frac{\beta-\gamma}2\quad\mbox{and}\quad n=2\,\frac{d-\gamma}{\beta+2-\gamma}\,,
\]
we claim that Inequality~\eqref{CKN} can be rewritten for a function $W$ such that
\[\label{notttt}
w(x)=W\(|x|^{\alpha-1}\,x\)\quad\forall\,x\in\R^d
\]
as
\[
\nrm W{2p,\delta}\le\mathsf K_{\alpha,n,p}\,\nrm{\DD W}{2,\delta}^\vartheta\,\nrm W{p+1,\delta}^{1-\vartheta}\,,\quad\forall\,W\in\mathrm H^p_{\delta,\delta}(\R^d)\,,
\]
with the notations
\[
\delta=d-n\,,\quad r=|x|\,,\quad\omega=\frac xr\,,\quad\DD W=\(\alpha\,\partial_r\!W,\,r^{-1}\,\nabla_{\kern-2pt\omega}W\)\,,
\]
where $\partial_r=\partial/\partial_r$ and $\nabla_\omega$ is the gradient in the angular derivatives, $\omega\in\sphere$.
The optimal constant $\mathsf K_{\alpha,n,p}$ is explicitly computed in terms of $\C_{\beta,\gamma,p}$ and the condition~\eqref{parameters} is equivalent to
\[
d\ge2\,,\quad\alpha>0\,,\quad n>d\quad\mbox{and}\quad p\in\(1,p_\star\right]\quad\mbox{with}\quad p_\star=\frac n{n-2}\,.
\]
By our change of variables, $w_\star$ is changed into
\[
W_\star(x):=\(1+|x|^2\)^{-1/(p-1)}\quad\forall\,x\in\R^d\,.
\]
The symmetry condition~\eqref{Symmetry condition} now reads
\[
\alpha\le\alpha_{\rm FS}\quad\mbox{with}\quad\alpha_{\rm FS}:=\sqrt{\frac{d-1}{n-1}}\,.
\]

For any $\alpha\ge 1$, note that the operator $\D$ can be rewritten as
\[
\D=\nabla+(\alpha-1)\,\frac x{|x|^2}\,(x\cdot\nabla)=\nabla+(\alpha-1)\,\omega\,\partial_r
\,.\]
If $\Dstar$ is the adjoint operator of $\D$, with respect to the measure $d\mu_n:=r^{n-1}\,dr\,d\omega$, then
\[
\Dstar Z=-\,|x|^\delta\,\nabla\cdot (|x|^{-\delta}\,Z)-\,(\alpha-1)\,r^{1-n}\,\omega\cdot\partial_r(r^{n-1}\,Z)
\]
for any vector-valued function $Z$ and moreover we have the useful formula
\[
\Dstar(W\,Z)=-\,\D W\cdot Z+W\,\Dstar Z
\]
if $W$ and $Z$ are respectively scalar- and vector-valued functions. Let us define the operator $\mathsf L_\alpha$ by
\[
\mathsf L_\alpha=-\,\Dstar\D=\alpha^2\(\partial^2_r+\frac{n-1}r\,\partial_r\)+\frac{\Delta_\omega}{r^2}\,,
\]
where $\Delta_\omega$ denotes the Laplace-Beltrami operator on $\sphere$.

We introduce the weighted equation
\[
\frac{\partial g}{\partial t}=\mathsf L_\alpha g^m\,,
\]
which is obtained from~\eqref{FD} by the change of variables
\[
v(t,x)=g\(t,|x|^{\alpha-1}\,x\)\quad\forall\,(t,x)\in\R^+\times\R^d\,.
\]
Next we use a self-similar change of variables similar to~\eqref{TDRS}, but with a scaling which corresponds to the artificial \emph{dimension} $n$. With $\mu=2+n\,(m-1)$ and $\kappa=\big(\frac{2\,m}{1-m}\big)^{1/\mu}$, let
\be{TDRS-weight}
g(t,x)=\frac1{\kappa^n\,R^n}\,u\!\(\tau,\frac x{\kappa\,R}\)\quad\mbox{where}\quad\left\{\begin{array}{l}
\frac{dR}{dt}=R^{1-\mu}\,,\quad R(0)=R_0=\kappa^{-1}\,,\\[6pt]
\tau(t)=\tfrac12\,\log\(\frac{R(t)}{R_0}\)\,.
\end{array}\right.
\ee
We observe that $\mu_\star=\alpha\,\mu$ with the notations of~\eqref{SelfSim} in Section~\ref{Sec:Main}.

In self-similar variables the function $u$ solves
\be{Eqn}
\frac{\partial u}{\partial\tau}=\Dstar(u\,z)
\ee
where
\[
z(\tau,x):=\D u^{m-1}-\,\frac2\alpha\,x=\D\(u^{m-1}-\frac{|x|^2}{\alpha^2}\)=\D\mathsf q\,,\quad\mathsf q:=u^{m-1}-\,\Balpha^{m-1}\quad\mbox{and}\quad\Balpha(x):=\(1+\frac{|x|^2}{\alpha^2}\)^\frac1{m-1}\,.
\]
The exponent $m$ is now in the range $m_1\le m<1$ with $m_1=1-1/n$. As in the case without weights, \emph{i.e.} the case $n=d$, we also consider the problem restricted to a ball $B_R$ and assume no-flux boundary conditions, that is,
\[
z\cdot\omega=0\quad\mbox{on}\quad\partial B_R\,.
\]
It is straightforward to check that
\[
\frac{\partial z}{\partial\tau}+\,(1-m)\,\D\,\Big(u^{m-2}\,\Dstar\(u\,z\)\Big)=0
\]
and, as a consequence,
\begin{multline*}
\frac d{d\tau}\iwR{u\,|z|^2}=\iwR{\frac{\partial u}{\partial\tau}\,|z|^2}+\,2\iwR{u\,z\cdot\frac{\partial z}{\partial\tau}}\\
=\iwR{\Dstar(u\,z)\,|z|^2}-\,2\,(1-m)\iwR{u\,z\cdot\D\,\Big(u^{m-2}\,\Dstar\(u\,z\)\Big)}\,.
\end{multline*}
Taking into account the boundary condition $z\cdot\omega=0$ on $\partial B_R$, a first integration by parts shows that
\[
\iwR{\Dstar(u\,z)\,|z|^2}=\iwR{u\,z\,\cdot\D\,|z|^2}\,.
\]
Hence we get
\begin{eqnarray*}
\frac d{d\tau}\iwR{u\,|z|^2}&\hspace*{-5pt}=&\hspace*{-5pt}\iwR{u\,z\,\cdot\D\,|z|^2}-\,2\iwR{u\,z\cdot\D\(z\cdot\D u^{m-1}+(1-m)\,u^{m-1}\,\Dstar z\)}\\
&\hspace*{-5pt}=&\hspace*{-5pt}\iwR{u\,z\,\cdot\D\,|z|^2}-\,2\iwR{u\,z\cdot\D\(|z|^2+\frac2{\alpha}\,z\cdot x+(1-m)\,u^{m-1}\,\Dstar z\)}\\
&\hspace*{-5pt}=&\hspace*{-5pt}-\iwR{u\,z\,\cdot\D\,|z|^2}-\,\frac4\alpha\iwR{u\,z\cdot\D(z\cdot x)}-\,2\,(1-m)\iwR{u\,z\cdot\D\(u^{m-1}\,\Dstar z\)}\,.
\end{eqnarray*}
Now, by expanding and integrating by parts we see that
\begin{multline*}
\iwR{u\,z\cdot\D\(u^{m-1}\,\Dstar z\)}=\iwR{\(u^m\,z\cdot\D\(\Dstar z\)+\frac{m-1}m\,\D u^m\cdot z\,\(\Dstar z\)\)}\\
=\frac1m\iwR{u^m\,z\cdot\D\(\Dstar z\)}-\frac{1-m}m\iwR{u^m\(\Dstar z\)^2}\,.
\end{multline*}
Integrating again by parts, we obtain
\[
\iwR{\D u^m\cdot\D|z|^2}=-\iwR{u^m\,\mathsf L_\alpha|z|^2}+\,\alpha\,R^{n-d}\int_{\partial B_R}u^m\,\omega\cdot\D|z|^2\,d\sigma\,.
\] 
So, after observing that $u\,z=\frac{m-1}m\,\D u^m-\frac2\alpha\,u\,x$, we finally get
\begin{multline*}\label{vraief}
\hspace*{-4pt}\frac d{d\tau}\iwR{u\,|z|^2}=\frac{m-1}m\iwR{u^m\(\mathsf L_\alpha|z|^2-\,2\,z\cdot\D\(\Dstar z\)+2\,(m-1)\(\Dstar z\)^2\)}\\
+\frac2\alpha\iwR{u\,x\,\cdot\D\,|z|^2}-\,\frac4\alpha\iwR{u\,z\cdot\D(z\cdot x)}\\
+\frac{1-m}m\,\alpha\,R^{n-d}\int_{\partial B_R}u^m\(\omega\cdot\D|z|^2\)\,d\sigma\,.
\end{multline*}
Next, since $z\cdot\omega=0$ on $\partial B_R$, exactly for the same reasons as in Section~\ref{Sec:RescaledVariables}, we know that
\[
\int_{\partial B_R}u^m\(\omega\cdot\D|z|^2\)\,d\sigma\le0\,.
\]

Let us define $\mathsf p:=u^{m-1}$ and
\[
\mathsf K[\mathsf p]:=\frac12\,\mathsf L_\alpha\,|\DD\mathsf p|^2-\,\DD\mathsf p\cdot\DD\mathsf L_\alpha\mathsf p-(1-m)\,(\mathsf L_\alpha\mathsf p)^2\,.
\]
By~\cite[Lemma~4.2]{DELM}, we know that
\[
\mathsf K[\mathsf p]=\alpha^4\(1-\tfrac1n\)\left[\mathsf p''-\frac{\mathsf p'}r-\frac{\Delta_\omega\,\mathsf p}{\alpha^2\,(n-1)\,r^2}\right]^2+\frac{2\,\alpha^2}{r^2}\left|\nabla_\omega\mathsf p'-\frac{\nabla_\omega\mathsf p}r \right|^2+\frac{\mathsf k[\mathsf p]}{r^4}+\(\tfrac1n-(1-m)\)\(\mathsf L_\alpha\mathsf q\)^2
\]
where
\[
\mathsf k[\mathsf p]:=\tfrac12\,\Delta_\omega\,|\nabla_\omega\mathsf p|^2-\nabla_\omega\mathsf p\cdot\nabla_\omega\Delta_\omega\,\mathsf p-\tfrac1{n-1}\,(\Delta_\omega\,\mathsf p)^2-(n-2)\,\alpha^2\,|\nabla_\omega\mathsf p|^2\,.
\]
As a consequence, we can write that
\[
\mathsf K[\mathsf p]=\mathsf K\left[\mathsf p-\Balpha^{m-1}\right]=\tfrac12\(\mathsf L_\alpha|z|^2-\,2\,z\cdot\D\(\Dstar z\)+2\,(m-1)\(\Dstar z\)^2\)\,.
\]
We also know that
\[
\isph{\mathsf k[\mathsf p]\,u^m}=\isph{\mathrm Q[\mathsf p]\,u^m}+(n-2)\(\alpha_{\rm FS}^2-\alpha^2\)\isph{|\nabla_\omega\mathsf p|^2\,u^m}\,,
\]
where, according to~\cite[Lemma~4.3]{DELM} (see details in the proof), for some explicit constants $a$ and $b$ which depend only on $\alpha$, $n$ and $d$, $\mathrm Q[\mathsf p]$ is such that
\[
\mathrm Q[\mathsf p]=\alpha_{\rm FS}^2\,\frac{n-2}{d-2}\,{\left\|(\nabla_\omega\otimes\nabla_\omega)\,\mathsf p-\frac1{d-1}\,(\Delta_\omega\mathsf p)+a\(\frac{\nabla_\omega\mathsf p\otimes\nabla_\omega\mathsf p}{\mathsf p}-\frac1{d-1}\,\frac{|\nabla_\omega\mathsf p|^2}{\mathsf p}\,g\)\right\|^2}+b^2\,\frac{|\nabla_\omega\mathsf p|^4}{|\mathsf p|^2}\quad\mbox{if}\quad d\ge3\,.
\]
Here $g$ denotes the standard metric on $\sphere$. The case $d=2$ has to be treated separately. According to~\cite[Lemma~4.3]{DELM}, there exists also an explicit constant, that we still denote by $b$, such that
\[
\isph{\mathrm Q[\mathsf p]\,u^m}\ge b^2\isph{\frac{|\nabla_\omega\mathsf p|^4}{\mathsf p^2}}\quad\mbox{if}\quad d=2\,.
\]
Collecting these observations, we have shown that
\begin{multline*}
\hspace*{-4pt}\frac d{d\tau}\iwR{u\,|z|^2}\\
\le-\,2\,\frac{1-m}m\iwR{\(\alpha^4\(1-\tfrac1n\)\left[\mathsf p''-\frac{\mathsf p'}r-\frac{\Delta_\omega\,\mathsf p}{\alpha^2\,(n-1)\,r^2}\right]^2+\frac{2\,\alpha^2}{r^2}\left|\nabla_\omega\mathsf p'-\frac{\nabla_\omega\mathsf p}r \right|^2+\(\tfrac1n-(1-m)\)\(\mathsf L_\alpha\mathsf q\)^2\)u^m}\\
-\,2\,\frac{1-m}m\iwR{\frac{\mathrm Q[\mathsf p]}{r^4}\,u^m}-\,2\,\frac{1-m}m\,(n-2)\(\alpha_{\rm FS}^2-\alpha^2\)\iwR{\frac{|\nabla_\omega\mathsf p|^2}{r^4}\,u^m}\\
+\frac2\alpha\iwR{u\,x\,\cdot\D\,|z|^2}-\,\frac4\alpha\iwR{u\,z\cdot\D(z\cdot x)}\,.
\end{multline*}
Since $x\cdot\D=\alpha\,r\,\partial_r$, $x\cdot z=\alpha\,r\partial_r\mathsf q$, $x\cdot\nabla_\omega=0$, and $z\cdot\partial_rz=z\cdot\partial_r(\D\mathsf q)=z\cdot\D\partial_r\mathsf q-\frac1{r^2}\,|\nabla_\omega\mathsf q|^2$, we have that
\begin{multline*}
\frac2\alpha\iwR{u\,x\,\cdot\D\,|z|^2}-\,\frac4\alpha\iwR{u\,z\cdot\D(z\cdot x)}\\
=\iwR{\(4\,r\,u\,z\cdot\partial_r z-\,4\,r\,u\,z\cdot\D\partial_r\mathsf q-\,4\,\alpha\,u\,\partial_r\mathsf q\,(\omega\cdot z)\)}\\
=-\,4\iwR{u\(\alpha^2\,|\partial_r\mathsf q|^2+\frac{\left|\nabla_\omega\mathsf q\right|^2}{r^2}\)}=-\,4\iwR{u\,|z|^2}\,.
\end{multline*}
After observing that $\(\tfrac1n-(1-m)\)\,(\mathsf L_\alpha\mathsf q)^2=(m-m_1)\(\mathsf L_\alpha u^{m-1}-2\,n\)^2$, we conclude that
\begin{multline*}
\frac d{d\tau}\iwR{u\,|z|^2}+4\iwR{u\,|z|^2}\\
\le-\,2\,\frac{1-m}m\iwR{\(\alpha^4\(1-\tfrac1n\)\left[\mathsf p''-\frac{\mathsf p'}r-\frac{\Delta_\omega\,\mathsf p}{\alpha^2\,(n-1)\,r^2}\right]^2+\frac{2\,\alpha^2}{r^2}\left|\nabla_\omega\mathsf p'-\frac{\nabla_\omega\mathsf p}r\right|^2\)u^m}\\
-\,2\,\frac{1-m}m\,(m-m_1)\iwR{\(\mathsf L_\alpha u^{m-1}-2\,n\)^2\,u^m}\\
-\,2\,\frac{1-m}m\iwR{\frac{\mathrm Q[\mathsf p]}{r^4}\,u^m}-\,2\,\frac{1-m}m\,(n-2)\(\alpha_{\rm FS}^2-\alpha^2\)\iwR{\frac{|\nabla_\omega\mathsf p|^2}{r^4}\,u^m}\,.
\end{multline*}
We can extend the function $u$ outside $B_R$ by the function $\Balpha$ and pass to the limit as $R$ goes to $+\infty$. If now we consider a  solution of~\eqref{Eqn} on $\R^d$ and if $\mathsf p:=u^{m-1}$, then we have
\begin{multline*}
\frac d{d\tau}\iwrd{u\,|z|^2}+4\iwrd{u\,|z|^2}\\
\le-\,2\,\frac{1-m}m\iwrd{\(\alpha^4\(1-\tfrac1n\)\left[\mathsf p''-\frac{\mathsf p'}r-\frac{\Delta_\omega\,\mathsf p}{\alpha^2\,(n-1)\,r^2}\right]^2+\frac{2\,\alpha^2}{r^2}\left|\nabla_\omega\mathsf p'-\frac{\nabla_\omega\mathsf p}r\right|^2\)u^m}\\
-\,2\,\frac{1-m}m\,(m-m_1)\iwrd{\(\mathsf L_\alpha u^{m-1}-2\,n\)^2\,u^m}\\
-\,2\,\frac{1-m}m\iwrd{\frac{\mathrm Q[\mathsf p]}{r^4}\,u^m}-\,2\,\frac{1-m}m\,(n-2)\(\alpha_{\rm FS}^2-\alpha^2\)\iwrd{\frac{|\nabla_\omega\mathsf p|^2}{r^4}\,u^m}\,.
\end{multline*}
This inequality implies~\eqref{CKN} in a non scale-invariant form (as in Section~\ref{Sec:RescaledVariables} when there are no weights), but also provides an additional integral remainder term. With $\P=\frac m{1-m}\,g^{m-1}$ and $v(t,x)=g\(t,r^{\alpha-1}x\)$, $r=|x|$, let us define
\begin{multline*}
\mathsf R[v]:=\iwrd{g^m\(\alpha^4\(1-\tfrac1n\)\left|\P''-\frac{\P'}r-\frac{\Delta_\omega\P}{\alpha^2\,(n-1)\,r^2}\right|^2+\frac{2\,\alpha^2}{r^2}\,\left|\nabla_\omega\P'-\frac{\nabla_\omega\P}r\right|^2\)}\\
+\iwrd{g^m\((n-2)\(\alpha_{\rm FS}^2-\alpha^2\)\frac{\left|\nabla_\omega\P\right|^2}{r^4}+\frac{b^2}{r^4}\,\frac{\left|\nabla_\omega\P\right|^4}{|\P|^2}\)}\,.
\end{multline*}
\begin{theorem}\label{Thm:Monotonicity} Let $d\ge2$. Under Condition~\eqref{parameters}, if
\[
\mbox{either}\quad\beta\le\beta_{\rm FS}(\gamma)\quad\forall\,\gamma\le0\,,\quad\mbox{or}\quad\gamma>0\,,
\]
then then there are two positive constants $\mathcal C_1$, $\mathcal C_2$, and a constant $b$, depending only on $\beta$, $\gamma$ and $d$ such that the following property holds.

Assume that~$v_0$ satisfies $\nrm{v_0}{1,\gamma}=M_\star$ and that there exist two positive constants $C_1$ and $C_2$ such that~\eqref{entredeuxBarenblatt} holds. Let us consider a positive solution of~\eqref{FD} with initial datum $v_0$ such that $v$ is \emph{smooth at $x=0$} for any $t\ge0$. Then, with the above notations, we have
\be{BestInequality}
\mathsf G[v(t,\cdot)]-\,\mathsf G[v_\star]\ge\mathcal C_1\int_t^\infty\mathsf h(s)^{3\,\mu-2}\,\mathsf R[v(s,\cdot)]\,ds\\
+\mathcal C_2\,\int_0^t\mathsf h(s)^{3\,\mu-2}\,\iwrd{\left|\mathsf L_\alpha\P-\,2\,n\,\mathsf h(s)^{-\mu}\right|^2}\;ds
\ee
for any $t\ge0$. Here $\mu$ and $\mathsf h$ are given by~\eqref{hmu}.\end{theorem}
The expressions of the constants $b$, $\mathcal C_1$ and $\mathcal C_2$ are explicit. See~\cite{DELM} for details. Since $\alpha<\alpha_{\rm FS}$ is equivalent to $\beta<\beta_{\rm FS}$ and
\[
\mathsf R[g]\ge(n-2)\(\alpha_{\rm FS}^2-\alpha^2\)\iwrd{g^m\,\frac{\left|\nabla_\omega\P\right|^2}{r^4}}\,,
\]
Theorem~\ref{Thm:Monotonicitylight} is a straightforward consequence of Theorem~\ref{Thm:Monotonicity}. In the opposite direction, by keeping all terms in $\mathsf Q[\mathsf p]$, it is possible to give a sharper estimate than~\eqref{BestInequality}, which has however no simple expression.

Under the above assumptions, Theorem~\ref{Thm:Rigidity} is a consequence of Theorem~\ref{Thm:Monotonicity}. Indeed, if we take $w^{2p}=v_0$, then we know that $\frac d{dt}\mathsf G[v(t,\cdot)]=0$ at $t=0$ because $v_0$ is a critical point of $\mathsf G$ under the mass constraint $\nrm{v_0}{1,\gamma}=M_\star$. Hence we know that
\[
0=\frac d{dt}\mathsf G[v(t,\cdot)]_{|t=0}\le-\,\mathsf R[v_0]\le0
\]
by differentiating~\eqref{BestInequality} at $t=0$, so that $\mathsf R[v_0]=0$, and this is enough to conclude. We can also notice that~\eqref{BestInequality} implies~\eqref{CKN} simply by dropping the right hand side and using a density argument, if necessary.

\begin{proof}[Proof of Theorem~\ref{Thm:Monotonicity}] For any solution $v$ of~\eqref{FD}, we apply the above computations with  $v(t,x)=g\(t,|x|^{\alpha-1}x\)$ and $u$ given by~\eqref{TDRS-weight}. Let us observe that
\[
\kappa\,R(t)=\mathsf h(t)
\]
if $R$ and $\mathsf h$ are given by~\eqref{hmu} and~\eqref{TDRS-weight}. It is then enough to undo the above changes of variables to obtain~\eqref{BestInequality}.\end{proof}

\section{Linearization and optimality}\label{Sec:Lin}

\subsection{The linearized fast diffusion flow and the spectral gap}

Let us perform a formal linearization of~\eqref{Eqn} around the Barenblatt profile $\Balpha$ by considering a solution $u_\varepsilon$ with mass $\ird{u_\varepsilon}=M_\star$ such that
\[
u_\varepsilon=\Balpha\(1+\varepsilon\,f\,\Balpha^{1-m}\)\,,
\]
and by taking formally the limit as $\varepsilon\to0$. We obtain that $f$ solves
\[
\frac{\partial f}{\partial t}=\mathcal L_\alpha\,f\quad\mbox{where}\quad\mathcal L_\alpha\,f:=(m-1)\,\Balpha^{m-2}\,\Dstar\(\Balpha\,\D f\)\,.
\]
We define the scalar products
\[
\scal{f_1}{f_2}=\iwrd{f_1\,f_2\,\Balpha^{2-m}}\quad\mbox{and}\quad\Scal{f_1}{f_2}=\iwrd{\D f_1\cdot\D f_2\,\Balpha}
\]
which correspond to weighted spaces, respectively of $\mathrm L^2$ and $\dot{\mathrm H}^1$ type. It is straightforward to check that the mass constraint results in the orthogonality condition
\[
\scal f1=0\,,
\]
that constant functions span the kernel of $\mathcal L_\alpha$, and that $\mathcal L_\alpha$ is self-adjoint on the space
\[
\mathcal X:=\mathrm L^2\big(\R^d,\,\Balpha^{2-m}\,d\mu_n\big)
\]
with norm given by $\|f\|^2=\scal ff$. Moreover,
\[
\frac12\,\frac d{dt}\scal ff=\scal f{\mathcal L_\alpha\,f}=\iwrd{f\,(\mathcal L_\alpha\,f)\,\Balpha^{2-m}}=-\iwrd{|\D f|^2\,\Balpha}=-\,\Scal ff
\]
if $f$ belongs to the subspace
\[
\mathcal Y:=\left\{f\in\mathrm L^2\big(\R^d,\,\Balpha^{2-m}\,d\mu_n\big)\,:\,\Scal ff<+\infty\right\}\,,
\]
and also
\[
\frac12\,\frac d{dt}\Scal ff=\iwrd{\D f\cdot\D(\mathcal L_\alpha\,f)\,\Balpha}=-\,\Scal f{\mathcal L_\alpha\,f}
\]
if $f$ is smooth enough.

Now let us consider the smallest positive eigenvalue $\lambda_1$ of $\mathcal L_\alpha$ on $\mathcal X$. An eigenfunction associated with $\lambda_1$ solves the eigenvalue equation
\[
-\,\mathcal L_\alpha\,f_1=\lambda_1\,f_1\,.
\]
According to~\cite{1602}, we know that $f_1\in\mathcal Y\subset\mathcal X$, so that
\[
\Scal{f_1}{f_1}=-\,\scal{f_1}{\mathcal L_\alpha\,f_1}=\lambda_1\,\scal{f_1}{f_1}
\]
and $f_1$ yields the equality case in the \emph{Hardy-Poincar\'e inequality}
\[
\Scal gg=-\,\scal g{\mathcal L_\alpha\,g}\ge\lambda_1\,\|g-\bar g\|^2\quad\forall\,g\in\mathcal Y\,.
\]
Here $\bar g:=\scal g1/\scal 11$ denotes the average of $g$. It turns out that
\[
\Scal{f_1}1=0\quad\mbox{and}\quad-\,\Scal{f_1}{\mathcal L_\alpha\,f_1}=\lambda_1\,\Scal{f_1}{f_1}
\]
so that $f_1$ is also optimal for the higher order inequality
\[
-\,\Scal g{\mathcal L_\alpha\,g}\ge\lambda_1\,\Scal gg
\]
written for any function $g\in\mathcal Y$ such that $\Scal g{\mathcal L_\alpha\,g}$ is finite. The proof of the inequality itself is a simple consequence of the expansion of the square
\[
-\,\Scal{(g-\bar g)}{\mathcal L_\alpha\,(g-\bar g)}=\scal{\mathcal L_\alpha\,(g-\bar g)}{\mathcal L_\alpha\,(g-\bar g)}=\|\mathcal L_\alpha\,(g-\bar g)\|^2\ge0\,.
\]
See~\cite{BBDGV,BDGV} and~\cite{MR1982656,MR2126633} for more details on the results in $\mathcal X$ and $\mathcal Y$ respectively. It has been observed in~\cite{BDGV} that the operator $\mathcal L_\alpha$ on $\mathcal X$ and its restriction to $\mathcal Y$ are unitarily equivalent when $(\alpha,n)=(1,d)$ and the extension to the general case is straightforward. The kernel of $\mathcal L_\alpha$ is generated by $f_0(x)=1$, and the eigenspaces corresponding to the next two eigenvalues are generated by $f_{1,k}(x)=x_k$ and $f_2(x)=|x|^2-c$, for some explicit constant $c$. If $(\beta,\gamma)=(0,0)$, the eigenvalues $\lambda_1$ and $\lambda_2$ are strictly ordered if $1-1/d<m<1$ and coincide if $m=1-1/d$, but the spectrum is more complicated in the general case: see~\cite[Appendix~B]{1602} for details. The key observation for our analysis is the fact that
\be{Lambda_FS}
\lambda_1\ge4\quad\Longleftrightarrow\quad\alpha\le\alpha_{\rm FS}:=\sqrt{\frac{d-1}{n-1}}\,.
\ee

\subsection{The optimality cases in the functional inequalities}

\subsubsection{Symmetry breaking in Caffarelli-Kohn-Nirenberg inequalities} It has been shown in~\cite{DELM} that the best constant in~\eqref{CKN} is determined by the infimum of
\[
\mathcal J[w]:=\vartheta\,\log\(\nrm{\D w}{2,\delta}\)+(1-\vartheta)\,\log\(\nrm w{p+1,\delta}\)-\log\(\nrm w{2p,\delta}\)
\]
with $\delta=d-n$ and that symmetry holds if $\mathcal J[w]\ge\mathcal J[w_\star]$ for any $w\in\mathrm H^p_{\delta,\delta}(\R^d)$, where $w_\star(x)=(1+|x|^2/\alpha^2)^\frac1{p-1}$. As a result of~\cite[Section~2.4]{1602}, we have that
\[
\mathcal J[w_\star+\varepsilon\,g]=\varepsilon^2\,\mathcal Q[g]+o(\varepsilon^2)\,,
\]
where
\begin{multline*}
\frac2\vartheta\,\nrm{\D w_\star}{2,\delta}^2\,\mathcal Q[g]=\nrm{\D g}{2,\delta}^2+\frac{p\,(2+\beta-\gamma)}{(p-1)^2}\,\big[d-\gamma-p\,(d-2-\beta)\big]\iwrd{|g|^2\,\frac1{1+\alpha^{-2}\,|x|^2}}\\
-p\,(2\,p-1)\,\frac{(2+\beta-\gamma)^2}{(p-1)^2}\iwrd{|g|^2\,\frac1{\(1+\alpha^{-2}\,|x|^2\)^2}}
\end{multline*}
is a nonnegative quadratic form if and only if $\alpha\le\alpha_{\rm FS}$ according to~\eqref{Lambda_FS}. Symmetry breaking therefore holds if $\alpha>\alpha_{\rm FS}$.

\subsubsection{An estimate on the information -- production of information inequality}\label{Sec:I-IP} As shown in Section~\ref{Sec:DirectWeigthed}, we define $\mathcal I[u]:=\iwrd{u\,|z|^2}$ where $z(\tau,x):=\D u^{m-1}-\,\frac2\alpha\,x$ and $\mathcal K[u]$ so that
\[
\frac d{d\tau}\mathcal I[u(\tau,\cdot)]=-\,\mathcal K[u(\tau,\cdot)]
\]
if $u$ solves~\eqref{Eqn}. If $\alpha\le\alpha_{\rm FS}$, then $\lambda_1\ge4$ and
\[
u\mapsto\frac{\mathcal K[u]}{\mathcal I[u]}-4
\]
is a nonnegative functional whose minimizer is achieved by $u=\Balpha$. With $u_\varepsilon=\Balpha\(1+\varepsilon\,f\,\Balpha^{1-m}\)$, we observe that
\[
4\le\mathcal C_2:=\inf_u\frac{\mathcal K[u]}{\mathcal I[u]}\le\lim_{\varepsilon\to0}\inf_f\frac{\mathcal K[u_\varepsilon]}{\mathcal I[u_\varepsilon]}=\inf_f\frac{\Scal f{\mathcal L_\alpha\,f}}{\Scal ff}=\frac{\Scal {f_1}{\mathcal L_\alpha\,f_1}}{\Scal {f_1}{f_1}}=\lambda_1\,.
\]
Of course one has to take some precautions ensuring that the mass is normalized and that denominator in the above quotients is never zero. Summarizing, what we observe is that the infimum of $\mathcal K/\mathcal I$ is achieved in the asymptotic regime as $u\to\Balpha$ and determined by the spectral gap of $\mathcal L_\alpha$ when $\lambda_1=4$, and that $\mathcal K/\mathcal I\ge4$ if $\lambda_1\ge4$, that is, when $\alpha=\alpha_{\rm FS}$ and $\alpha\le\alpha_{\rm FS}$ respectively.

\subsubsection{Symmetry in Caffarelli-Kohn-Nirenberg inequalities}\label{Sec:SymCKN} If $\alpha\le\alpha_{\rm FS}$, the fact that $\mathcal K/\mathcal I\ge4$ has an important consequence. Indeed we know that
\[
\frac d{d\tau}\(\mathcal I[u(\tau,\cdot)]-\,4\,\mathcal E[u(\tau,\cdot)]\)\le0
\]
where
\[
\mathcal E[u]:=-\,\frac1m\iwrd{\(u^m-\,\Balpha^m-\,m\,\Balpha^{m-1}\,(u-\,\Balpha)\)}\,,
\]
so that
\[
\mathcal I[u]-\,4\,\mathcal E[u]\ge\mathcal I[\Balpha]-\,4\,\mathcal E[\Balpha]=0\,.
\]
This inequality is equivalent to $\mathcal J[w]\ge\mathcal J[w_\star]$, which establishes that optimality in~\eqref{CKN} is achieved among symmetric functions. In other words, the computations of Section~\ref{Sec:DirectWeigthed} show that for $\alpha\le\alpha_{\rm FS}$, the function
\[
\tau\mapsto\mathcal I[u(\tau,\cdot)]-\,4\,\mathcal E[u(\tau,\cdot)]
\]
is monotone decreasing. The condition $\alpha\le\alpha_{\rm FS}$ is a sufficient condition, which is however complementary of the symmetry breaking condition (because it depends only on the sign of $\lambda_1-4$), and this explains why the method based on nonlinear flows provides the optimal range for symmetry.

\subsubsection{Optimality of the information -- production of information inequality} From Section~\ref{Sec:I-IP}, we know that the infimum of $\mathcal K/\mathcal I$ is achieved in the asymptotic regime as $u\to\Balpha$ and determined by the spectral gap of $\mathcal L_\alpha$ when $\lambda_1=4$. This covers in particular the case without weights of the Gagliardo-Nirenberg inequalities~\eqref{GN} and of the fast diffusion equation~\eqref{poro} studied in Section~\ref{Sec:GN}.

We also know that
\[
\mathcal C_2=\inf_u\frac{\mathcal K[u]}{\mathcal I[u]}\le\lambda_1<4
\]
if $\alpha>\alpha_{\rm FS}$, and that
\[
\lambda_1\ge\mathcal C_2=\inf_u\frac{\mathcal K[u]}{\mathcal I[u]}>4
\]
if $\alpha<\alpha_{\rm FS}$. The inequality is strict because, otherwise, if $4$ was optimal, it would be achieved in the asymptotic regime and therefore would be equal to $\lambda_1>4$, a contradiction.

\subsubsection{Optimality of the entropy -- production of entropy inequality} Arguing as in Section~\ref{Sec:SymCKN}, we know that
\[
\mathcal I[u]-\,\mathcal C_2\,\mathcal E[u]\ge\mathcal I[\Balpha]-\,\mathcal C_2\,\mathcal E[\Balpha]=0\,.
\]
As a consequence, we have that
\[
\mathcal C_1:=\inf_u\frac{\mathcal I[u]}{\mathcal E[u]}\ge\mathcal C_2=\inf_u\frac{\mathcal K[u]}{\mathcal I[u]}\,.
\]
With $u_\varepsilon=\Balpha\(1+\varepsilon\,f\,\Balpha^{1-m}\)$, we observe that
\[
4\le\mathcal C_1\le\lim_{\varepsilon\to0}\inf_f\frac{\mathcal I[u_\varepsilon]}{\mathcal E[u_\varepsilon]}=\inf_f\frac{\scal f{\mathcal L_\alpha\,f}}{\scal ff}=\frac{\scal {f_1}{\mathcal L_\alpha\,f_1}}{\scal {f_1}f_1}=\lambda_1=\lim_{\varepsilon\to0}\inf_f\frac{\mathcal K[u_\varepsilon]}{\mathcal I[u_\varepsilon]}\,.
\]

If $\alpha=\alpha_{\rm FS}$, then $\lambda_1=4=\mathcal C_1=\mathcal C_2$. Again this covers in particular the case without weights of the Gagliardo-Nirenberg inequalities~\eqref{GN} and of the fast diffusion equation~\eqref{poro}. If $\alpha<\alpha_{\rm FS}$, then $\mathcal C_1\ge\mathcal C_2>4$. Conversely, if $\alpha>\alpha_{\rm FS}$, then $\mathcal C_1\le\lambda_1<4$. We know from~\cite{1602} that $\mathcal C_1>0$, and also that the optimal constant is achieved, but the precise value of $\mathcal C_1$ is so far unknown.


\medskip\noindent{\bf Acknowledgment:} This work has been partially supported by the Projects STAB and Kibord (J.D.) of the French National Research Agency (ANR). M.L. has been partially supported by NSF Grant DMS- 1600560 and the Humboldt Foundation. Part of this work was done at the Institute Mittag-Leffler during the fall program \emph{Interactions between Partial Differential Equations \& Functional Inequalities}.\\
\noindent{\copyright\,2016 by the authors. This paper may be reproduced, in its entirety, for non-commercial purposes.}


\bigskip\begin{center}
\rule{2cm}{0.5pt}
\end{center}\bigskip

\end{document}